\documentclass[11pt,table]{amsart}
\usepackage{tikz}
\usetikzlibrary{positioning}
\usepackage{fullpage,url,amssymb,enumerate,colonequals}
\usepackage{multirow}
\usepackage{float}
\usepackage{mathrsfs} 
\usepackage[section]{placeins}
\usepackage{MnSymbol}
\usepackage{extarrows}
\usepackage{lscape}
\usepackage[all,cmtip]{xy}
\usepackage{longtable}
\usepackage[OT2,T1]{fontenc}

\usepackage{comment}
\usepackage{multirow}
\usepackage{hyperref}

\newcommand{\Q}{\mathbb{Q}}

\DeclareMathOperator{\NP}{NP}

\newcommand{\isom}{\simeq}

\newcommand{\union}{\cup} 

\newcommand{\floor}[1]{\left\lfloor #1 \right\rfloor}

\numberwithin{equation}{section}

\newtheorem{thm}{Theorem}
\newtheorem{lem}[thm]{Lemma}
\newtheorem{prop}[thm]{Proposition}
\newtheorem{cor}[thm]{Corollary}

\theoremstyle{definition}
\newtheorem{defn}[thm]{Definition}
\newtheorem{eg}[thm]{Example}

\newtheorem*{Ack}{Acknowledgements}

\theoremstyle{remark}

\theoremstyle{remark}
\newtheorem{remark}[equation]{Remark}

\definecolor{darkgreen}{rgb}{0,0.5,0}

\DeclareRobustCommand{\SkipTocEntry}[5]{}

\begin{document}
\title{Intersections of the Ekedahl-Oort and Newton strata of $\mathcal{A}_{5}$}

\begin{abstract}
The moduli space $\mathcal{A}_g$ of principally polarised abelian varieties of dimension $g$, defined over an algebraically closed field of characteristic $p >0$, is studied through various stratifications. The two most prominent ones are the Newton stratification, based on the isogeny class of the $p$-divisible group of an abelian variety, and the Ekedahl-Oort stratification, defined by the isomorphism class of its $p$-torsion group scheme. In general, it is not known how the strata of these two intersect. In this paper we completely determine which of these intersections are non-empty in dimension five. As a consequence, we give an explicit description of the induced Ekedahl-Oort stratification on the supersingular locus $\mathcal{S}_{5}$. 
\end{abstract}

\author{Steven R.  Groen}
 \author{Elvira Lupoian} 
 \author{Mychelle Parker}

\address{Steven Groen \\ Korteweg-de Vries Institute for Mathematics, University of Amsterdam, Amsterdam, Netherlands}
\email{s.r.groen@uva.nl}

\address{Elvira Lupoian \\ Department of Mathematics, University College London, London \\ United Kingdom}
 \email{e.lupoian@ucl.ac.uk}

 \address{Mychelle Parker \\ Department of Mathematics, The University of Texas at Austin, Austin  \\ United States}
 \email{mychelle.parker@austin.utexas.edu}

 
\maketitle

\section{Introduction}
Let $p$ be a prime number, $k$ an algebraically closed field of characteristic $p$, $ g \ge 1$ an integer and consider $\mathcal{A}_g \otimes k $, the moduli space of principally polarised abelian varieties of dimension $g$ defined over $k$, henceforth denoted by $\mathcal{A}_g$. The most natural method for studying $\mathcal{A}_g$ appears to be via the \textit{Ekedahl-Oort stratification}, first introduced by Oort \cite{oort2001stratification}, and based on understanding the $p$-torsion group scheme $X[p]$ of an abelian variety $X$ in characteristic $p$ through the action of Frobenius and Verschiebung on it. These actions are classified by an \textit{elementary sequence}, an isomorphism invariant of the abelian variety, with each Ekedahl-Oort stratum containing abelian varieties with a specific elementary sequence. A natural question is how this stratification restricts to important sub-schemes of $\mathcal{A}_g$. Perhaps, the most well-studied subscheme of $\mathcal{A}_g$ is the \textit{supersingular locus}, defined as
\begin{center}
    $\mathcal{S}_{g} := \{ (X, \lambda) \in \mathcal{A}_{g} \ \vert \ X  \ \text{is supersingular } \}$.
\end{center}

Supersingular abelian varieties have the same Newton polygon, the constant line segment of slope $\frac{1}{2}$. This allows us to view the supersingular locus as a \textit{Newton stratum}. The \textit{Newton stratification} of $\mathcal{A}_g$ is based on the isogeny class of the $p$-divisible group $X[p^{\infty}]$ of an abelian variety $X$. Given the close apparent relation between the quantities defining these stratifications, it is natural to ask how an Ekedahl-Oort stratum and Newton stratum in $\mathcal{A}_g$ intersect.  In dimension at most three, intersections are straightforward to determine due to the small number of strata. The intersections in dimension four (and $p$-rank zero) were classified in the work of Ibukiyama, Karemaker and Yu \cite{ibukiyama2022polarised}. In this paper we study intersections in dimension five, yielding in particular the Ekedahl-Oort stratification of $\mathcal{S}_5$.

\subsection{Main Result}
Throughout this section, we write $\mathcal{N}(\xi)$ for the Newton stratum of $\mathcal{A}_g$ consisting of abelian varieties whose Newton polygon is $\xi$ (see Section \ref{npstrat} for definitions).
We write $S_{\varphi}$ for Ekedahl-Oort stratum of $\mathcal{A}_g$ consisting of abelian varieties whose elementary sequence is $\varphi$ (see Section \ref{eostrat} for definitions).

\begin{thm} \label{mainthm}
  The intersections of Ekedahl-Oort strata and Newton strata of $\mathcal{A}_5$ are described in Table \ref{intstable}. All intersections are non-empty.
\end{thm}

\begin{table}
    \centering
    \begin{tabular}{|c|c|c|c|}
    \hline
    $g$ & $ f$ & $\xi$ & $ \varphi \ \text{such that} \ S_{\varphi} \cap \mathcal{N}(\xi) \ne \emptyset$ \\
    \hline 
    $5$ & $5$ & $5([1,0] + [0,1]) $ & $(1,2,3,4,5)$  \\
    \hline 
    $5$ & $4$ & $4([1,0] + [0,1]) + [1,1]$ &  $(1,2,3,4,4) $  \\
    \hline 
   $5$ & $3$ & $3([1,0] + [0,1]) + 2[1,1]$ & $(1,2,3,3,3), (1,2,3,3,4)$  \\
   \hline 
   $5$ & $2$ & $2[1,0] + [2,1] + [1,2] + 2[0,1]$ & $(1,2,2,3,3), (1,2,2,3,4) $  \\
    \hline 
    $5$ & $2$ & $2[1,0] + 3[1,1] + 2[0,1]$ & $(1,2,2,2,2), (1,2,2,2,3), (1,2,2,3,4)$   \\
    \hline 
    $5$ & $1$ & $[1,0] + [3,1] + [1,3] + [0,1]$ & $(1,1,2,3,3), (1,1,2,3,4)$  \\
    \hline 
    $5$ & $1$ & $[1,0] + [2,1] + [1,1] + [1,2] + [0,1]$ & $(1,1,2,2,2), (1,1,2,2,3), (1,1,2,3,4) $  \\
    \hline 
    $5$ & $1$ & $[1,0] + 4[1,1] + [0,1]$ & $(1,1,1,1,1),(1,1,1,1,2),(1,1,1,2,2),(1,1,1,2,3), $ \\ & & & $ (1,1,2,2,3), (1,1,2,3,4) $ \\ 
    \hline 
   $5$ & $0$ & $[4,1] + [1,4]$ & $(0,1,2,3,3),(0,1,2,3,4)$  \\
    \hline 
    $5$ & $0$ & $[3,1] + [1,1] + [1,3] $ & $(0,1,2,2,2),(0,1,2,2,3), (0,1,2,3,4)$  \\
    \hline
     $5$ & $0$ & $[2,1] + 2[1,1] + [1,2] $ & $(0,0,1,2,2),(0,0,1,2,3), (0,1,1,1,1), (0,1,1,1,2), $ \\  & & & $ (0,1,1,2,2), (0,1,1,2,3),  (0,1,2,2,3), (0,1,2,3,4) $  \\
     \hline 
     $5$ & $0$ & $[3,2] + [2,3] $ & $(0,0,1,1,1),(0,0,1,1,2),(0,0,1,2,3), (0,1,1,2,2),$ \\ & & &  $ (0,1,1,2,3),     (0,1,2,2,3), (0,1,2,3,4),$  \\
\hline 
$5$ & $0 $ & $5[1,1] $ & $(0,0,0,0,0), (0,0,0,0,1), (0,0,0,1,1),(0,0,0,1,2),$ \\ & & &  $(0,0,1,1,2), (0,0,1,2,3), (0,1,1,1,2), (0,1,1,2,2), $ \\  & & &  $(0,1,1,2,3), (0,1,2,2,3),  (0,1,2,3,4) $  \\
    \hline 

 \end{tabular}
      \caption{Intersections of the Ekedahl-Oort strata and Newton strata in dimension $5$.}
    \label{intstable}
    \end{table}
In particular, we recover the Ekedahl-Oort stratification of the supersingular locus. 
\begin{cor} The Ekedahl-Oort stratification of the supersingular locus is
 \begin{align*}
      \mathcal{S}_{5} = & S_{(0,0,0,0,0)} \cup S_{(0,0,0,0,1)} \cup S_{(0,0,0,1,1)} \cup S_{(0,0,0,1,2)} \cup (S_{(0,0,1,1,2)} \cap \mathcal{S}_{5}) \cup (S_{(0,1,1,2,3)}  \cup \mathcal{S}_{5}) \cup (S_{(0,0,1,2,3)} \cap \mathcal{S}_{5}) \\ &  \cup (S_{(0,1,1,2,2)} \cap \mathcal{S}_{5}) \cup (S_{(0,1,1,1,2)} \cap \mathcal{S}_{5})  \cup (S_{(0,1,2,2,3)} \cap  \mathcal{S}_{5}) \cup ( S_{(0,1,2,3,4)} \cap \mathcal{S}_{5})
      \end{align*}
      where all intersections appearing above are non-empty.
\end{cor}

The interaction between the Ekedahl-Oort stratification and the Newton stratification of $\mathcal{A}_5$ exhibits interesting behaviour that is not seen in lower dimension; see Remarks~\ref{rmk: TL not highest}, \ref{rmk: superspecial} and \ref{rmk: Newton sandwich}.

\subsection{Methodology}

The $p$-rank of an abelian variety $X/k$, is the unique integer $0 \le f(X) \le \text{dim}(X)$ such that $ \vert X[p](k) \vert = p^{f(X)} $. The \textit{$p$-rank f stratum} of $ \mathcal{A}_g$ is the subscheme consisting of all abelian varieties  of $p$-rank equal to $f$. The $p$-rank stratification is refined by both the Ekedahl-Oort and the Newton polygon stratifications (c.f Propositions \ref{pranksofnewtonpolygons} and \ref{EOfacts}) and thus it provides a natural stepping stone for studying intersections of the strata of the two. Moreover,  intersections between strata of positive $p$-rank reduce to those of lower dimension and $p$-rank zero, as explained in Section \ref{positivesections}, and thus this paper will largely focus on the $p$-rank zero case.

The main challenges of classifying intersections in dimension five and $p$-rank zero arise from the large number of possible intersections and limitations of current theoretical results in treating all possibilities. For instance, in the case of $\mathcal{S}_4$ there are eight Ekedahl-Oort strata which could intersect $\mathcal{S}_4$,  four of which are known to be fully contained in it by a classical result of Chai-Oort, two are ruled out using minimal $p$-divisible groups, which leaves two possible intersections to consider. A  key tool used to understand the remaining intersections is the $a$-number stratification of $\mathcal{S}_4$ and results of Harashita \cite{harashita2004number}; with this method being applicable only when each $a$-number stratum can intersect at most two strata. In particular, this cannot be used to decide any intersections in our problem. In dimension five, our strategy for classifying intersections is as follows.
\begin{itemize}
\item[1.] We use \textit{minimal $p$-divisible groups} and Proposition \ref{contaimentinsslocus} to find the Ekedahl-Oort strata which are fully contained in a Newton stratum.
\item[2.] We use \textit{first slopes} to deduce that certain intersections are non-empty, following Harashita \cite{harashitaslope}.
\item[3.] We compute elementary sequences attached to \textit{products} of abelian varieties, and use results from \textit{lower dimensions} to conclude that some intersections are non-empty. 
    \item[4.] To decide remaining intersections, we construct \textit{explicit Dieudonn{\'e} modules}, or prove that the relevant Dieudonn{\'e} modules do not exist.  
\end{itemize}
The \texttt{magma} code used for computations in this paper is publicly available in the \texttt{Github} repository:
\url{https://github.com/ElviraLupoian/StratificationIntersectionsInDimension5}

\subsection{Outline} This paper is organised as follows. In Section \ref{stratsdefs} we recall definitions of the stratifications appearing in this work. In Section \ref{intsec} we collect several results from literature which allow us to deduce that certain intersections are non-empty. 
In Section \ref{positivesections} we treat intersections of positive $p$-rank. In Section~\ref{sec: explicit} we discuss explicit Dieudonn\'{e} modules and in Section~\ref{prankzeroproof} we conclude the proof of Theorem~\ref{mainthm}.
  \begin{Ack}
This project began at the Arizona Winter school 2024 and we are grateful to the organisers. We thank Valentijn Karemaker for suggesting the project, for many helpful suggestions  throughout and for her careful reading of this paper. We thank Soumya Sankar for her help in Arizona and after. We thank Leo Mayer for his initial contribution to this project. The second  named author thanks Damiano Testa  for many helpful conversations, and Rachel Pries for her comments regarding explicit constructions. The first author is supported by grant VI.Vidi.223.028 of the Dutch Research Council (NWO). 
The second author is supported by the EPSRC Doctoral Prize fellowship EP/W524335/1 and the extension UKRI3030.
 \end{Ack}

\section{Stratifications of  $\mathcal{A}_{g}$} \label{stratsdefs}

Throughout this paper, $k$ is an algebraically closed field of characteristic $p$ and $X$ denotes an abelian variety of dimension $g$ over $k$.
\subsection{Invariants of abelian varieties}

We now discuss invariants of $X$ that will be important later.

\begin{defn}  The \textit{$p$-rank} of $X$ is an integer $0 \le f(X) \le g$ such that $\vert X[p](k) \vert = p^{f(X)}$.  
\end{defn}
  For any $0 \le f \le g$, the \textit{$p$-rank $f$ stratum} is 
     \begin{center}
         $V_f = \{ [X, \lambda] \in \mathcal{A}_g \ \vert \ f(X) = f \}$.
   \end{center}
 The $p$-rank strata are locally closed subschemes of $\mathcal{A}_g$ and they form a good stratification of $\mathcal{A}_g$.
\begin{defn}
    The $a$-number of $X$, $a(X)$, is the dimension of $\text{Hom}_{k}(\alpha_p, X)$, where $\alpha_p$ denotes the group scheme $\mathrm{Spec}(k[x]/(x^p))$.
 \end{defn}
 For any $0 \le a \le g$, the \textit{$a$-number stratum} of $\mathcal{A}_g$  defined by $a$ is \begin{center}
     $\mathcal{A}_g(a) = \{ [ X, \lambda] \in \mathcal{A}_g \ \vert \  a(X) = a\}. $
 \end{center}
 \subsection{Dieudonn{\'e} Theory} \label{ss: Dd}
 The two main stratifications studied in this paper are defined using finite group schemes and $p$-divisible groups. These can be classified using Dieudonn{\'e} modules, which we briefly review in this section, and we refer the reader to \cite[Section 15.3]{oort2001stratification} and references thereafter for precise details. 
 
 Let $W(k)$ be the ring of infinite Witt vectors with coordinates in $k$. We write
  \begin{center}
      $E_k := W(k)[ F, V]/ \langle Fa  - a^{\sigma} F, V a^{\sigma} - a V, FV - p, VF - p \ \vert \ a \in W(k) \rangle $
  \end{center}
 where $\sigma$ is the Frobenius morphism on $W(k)$, that is a lift of the Frobenius map $x \mapsto x^p$ on $k$.
 A \textit{Dieudonn{\'e} module} over $W(k)$ is a left $E_{k}$-module which is finitely generated as a $W(k)$-module. 

There are canonical categorical equivalences, both denoted by $\mathbb{D}$, from the category of $p$-torsion finite commutative group schemes (resp. $p$-divisible groups) over $k$ to the category of Dieudonn{\'e} modules over $W(k)$ that are of finite $W(k)$-length and are annihilated by $p$ (resp. that are free as $W(k)$-modules).

If $(X,\lambda)$ is a principally polarised abelian variety, then $\lambda$ restricts to a principal quasi-polarisation on the $p$-divisible group $X[p^\infty]$. Through the Dieudonn\'{e} functor, this gives rise to a non-degenerate, alternating, bilinear pairing $\langle \cdot , \cdot \rangle$ on $\mathbb{D}(X[p^\infty])$ satisfying $\langle Fx, y\rangle = \langle x, Vy\rangle^\sigma$ for every $x,y\in\mathbb{D}(X[p^\infty])$. We call a Dieudonn\'{e} module \textit{principally quasi-polarised} if it admits such a pairing. 

\subsection{Newton stratification} \label{npstrat}
The Newton stratification is defined using the $p$-divisible group $X[p^{\infty}]$ attached to an abelian variety $X/k$, and the associated Dieudonn{\'e} module, which we briefly review in this subsection, see \cite[Section 15.5]{oort2001stratification} for details.

Let $m,n$ be relatively prime, positive integers. The $p$-divisible group $G_{m,n}$ over $\mathbb{F}_p$ is defined by
$$\mathbb{D}(G_{m,n}) = E_{\mathbb{F}_p}/E_{\mathbb{F}_p}(F^m - V^n).$$
It follows from \cite[Section 4.3, Theorem 4.1]{yuri1963manin} that any $p$-divisible group $G$ over $k$ is isogenous to a product of the form 
\begin{center}
    $ G \sim_{k} \displaystyle \bigoplus_{i=1}^{t} G_{m_i, n_i}$
\end{center}
for a finite set $\{ (m_i, n_i) : 1 \le i \le t \}$ of pairs of co-prime positive integers.  We define the \textit{Newton polygon} $\NP(G)$ of $G$ defined as follows. We order the indices of  $\bigoplus_{i=1}^{t} G_{(m_i, n_i)}$, such that $\lambda_{i} = \frac{n_{i}}{m_{i} + n_{i}}$ for $ 1 \le i \le t$, are non-decreasing, that is,  $\lambda_{1} \le \ldots \le \lambda_{t}$. The \textit{Newton polygon} $\NP(G) = \sum_{i=1}^{t}[m_{i}, n_{i}]$ is the piecewise linear function which starts at $(0,0)$, ends at $(\sum_{i=1}^{t}(m_{i} + n_{i}), \sum_{i=1}^{t} n_{i})$ and has $(\sum_{i=1}^{j}(m_{i} + n_{i}), \sum_{i=1}^{j} n_{i})$ as its break points for $1 \le j \le t$.

In the case where $G = X[p^{\infty}]$ for some $x = (X, \lambda) \in \mathcal{A}_g$, the decomposition is symmetric, that is, 
\begin{center}
    $X[p^{\infty}] \sim_{k} \displaystyle \bigoplus_{i=1}^{t}\left(G_{a_i,b_i} \oplus G_{b_i, a_i}\right) \oplus G_{1,1}^{\oplus s} \oplus \left(G_{1,0} \oplus G_{0,1}\right)^{\oplus f}$
\end{center}
which is known as the \textit{formal isogeny type} of $X$. In terms of the Newton polygon, this means that the slopes $\lambda$ and $1-\lambda$ occur with the same multiplicity in $\NP(X[p^\infty])$ and we call such a Newton polygon \textit{symmetric}.

\begin{defn}
  The \textit{Newton polygon} of an abelian variety $X$ is $\NP(X) = \NP (X[p^{\infty}])$.
\end{defn}
\begin{defn}
Let $\xi$ be a symmetric Newton polygon from $(0,0)$ to $(2g,g)$. The \textit{Newton stratum} in $\mathcal{A}_{g}$ defined by $\xi$ is 
\begin{center}
    $\mathcal{N}(\xi) := \{ (X, \lambda) \in \mathcal{A}_{g} \ \vert \ \NP(X) = \xi \}$.
\end{center}
\end{defn}
Newton strata are non-empty (\cite[Page 98]{MR3077121}), they define a good stratification of $\mathcal{A}_{g}$, and for all $\xi \ne g[1,1]$, $\mathcal{N}(\xi) $ is geometrically irreducible (\cite[Theorem 3.1]{chai2011monodromy}). Moreover, the Newton stratification refines the $p$-rank stratification by the following. 
\begin{prop}\label{pranksofnewtonpolygons}
The $p$-rank of an abelian variety is the number of zero slopes in its Newton polygon.    
\end{prop}
\begin{proof}
 As argued in \cite[Page 14, Remark 2.29]{karemaker2024geometry}, this follows from \cite[Remark 3.3]{oort1991moduli}.
\end{proof}
\begin{defn}
    The \textit{supersingular locus} of $\mathcal{A}_g$ is 
\begin{center}
    $\mathcal{S}_{g} := \{ (X, \lambda) \in \mathcal{A}_{g} \ \vert \ X  \ \text{is supersingular } \}.$
\end{center}
\end{defn}
An abelian variety $X$ is \textit{supersingular} if its Newton polygon is $\sigma_{g} = g[1,1]$. Therefore, the supersingular locus $\mathcal{S}_g$ is the Newton stratum $\mathcal{N}(\sigma_g)$.

\subsection{Ekedahl-Oort Stratification} \label{eostrat}
The Ekedahl-Oort stratification of $\mathcal{A}_g$ is based on the $p$-torsion subgroup schemes of its points. These are in turn related to Dieudonn{\'e} modules using $\text{BT}_1$s. 
\begin{defn}
A \textit{Barsotti-Tate truncated level $1$ group scheme} ( $\text{BT}_{1}$) over $k$, is a finite commutative group scheme $G$, defined over $k$ which satisfies:
\begin{itemize}
    \item[(i)] $\text{Im}( \mathrm{Ver} : G^{(p)} \rightarrow G) = \text{Ker}(\mathrm{Frob} : G \rightarrow G^{(p)});$
    \item[(ii)] $\text{Im}( \mathrm{Frob} : G \rightarrow G^{(p)}) = \text{Ker}(\mathrm{Ver} : G^{(p)} \rightarrow G),$ 
\end{itemize}
where $\mathrm{Frob}$ and $\mathrm{Ver}$ are the Frobenius and Verschiebung morphisms. 
\end{defn}

\begin{defn}
If $N$ is a $BT_1$ over $k$, we say $N$ is \textit{polarised} if $\mathbb{D}(N)$ is principally quasi-polarised. 

\end{defn}
\begin{defn} \label{esdefn}
 An \textit{elementary sequence} of length $g$ is a function  $\varphi : \lbrace 0 , \ldots, g \rbrace \longrightarrow \{ 0, \ldots, g\}$, satisfying: 
\begin{itemize}
    \item[(i)] $\varphi \left( 0 \right) = 0$;
    \item[(ii).] $\varphi \left( i \right) \le \varphi \left( i +1 \right) \leq \varphi \left( i \right) +1$ for all $0 \le i < g$.
\end{itemize}
We write  $\varphi$ as $(\varphi(1), \ldots, \varphi(g))$. The \textit{final sequence} stretched from $\varphi$ is the function $\psi : \{ 0, \ldots, 2g \} \rightarrow \{ 0, \ldots, 2g \}$ defined by $ \psi(2g -i) = g -i + \varphi(i)$ and $ \psi(i) = \varphi(i)$ for all $ 0 \le i \le g$.\end{defn}
Let $g \ge 1$ be an integer. We define a canonical map, $\text{ES}$, from the set of polarised $\text{BT}_{1}$s of length $2g$ over $k$, up to isomorphism, to the set of elementary sequences of length $g$, as follows. 
Let $G$ be a polarised $\text{BT}_{1}$ of length $2g$, defined over $k$. For any $k$-subgroup scheme $H \subseteq G$, and any finite word $w$ in $\mathrm{Ver}$ and $\mathrm{Frob}^{-1}$, we define $w \cdot H$ inductively by:
\begin{center}
    $\mathrm{Ver} \cdot H = \mathrm{Ver} ( H^{(p)}) $ and $ \mathrm{Frob}^{-1} \cdot H = \mathrm{Frob}^{-1}( H^{(p)} \cap \mathrm{Frob}(G))$.
\end{center}
We let $\psi : \{ 0, \ldots, 2g \} \rightarrow \{ 0, \ldots, 2g \}$ be the unique final sequence satisfying
\begin{center}
    $ \psi( \text{length}(w \cdot G)) = \text{length}( \mathrm{Ver} \cdot w \cdot G))$
\end{center}
any finite  word  $w$, and we set $\varphi(i) = \psi (i) $ for all $0 \le i \le g$. We define $\text{ES}(G) := \varphi $.

Conversely, for an elementary sequence $ \varphi$ of length $g$, we define a polarised $BT_1$, $(N_{\varphi}, \langle, \rangle_{\varphi})$ over $\mathbb{F}_p$ with $\text{ES}(N_{\varphi}) = \varphi$. For the rest of this subsection, we follow \cite[Section 9]{oort2001stratification} closely. Let $\psi$  be the final sequence stretched from $\varphi$ and write 
$$ 1 \le m_1 < \ldots  < m_g \le 2g $$
for the sequence of indices where $\psi$ jumps, that is, $\psi (i -1 ) < \psi (i)$, and 
$$ 1 \le n_g < \ldots < n_1 \le 2g$$
for the sequence of indices where $\psi$ is constant. Then $A_{\varphi} = \mathbb{D}(N_{\varphi})$ is the vector space over $\mathbb{F}_p$ with basis $\{ Z_1, \ldots, Z_{2g} \}$ defined by 
$$Z_{m_i} = X_i  \ \text{and} \ Z_{n_i} = Y_i \ \ \text{for} \ 1 \le i \le g.$$
We further define $ F: A_{\varphi} \rightarrow A_{\varphi}$ and $F : A_{\varphi} \rightarrow A_{\varphi}$ by 
$$ F(X_i) = Z_i, \; \; \; \; F(Y_i) =0, \; \; \; \; F(Z_i) =0, \; \; \; \; F(Z_{2g - i +1}) = \begin{cases}
        Y_i &  \text{if} \ Z_{2g - i + 1} \in \{ Y_g, \ldots, Y_1 \}; \\  
        - Y_i  & \text{if} \ Z_{2g - i + 1} \in \{ X_1, \ldots, X_g \} .
     \end{cases}$$
for all $ 1 \le i \le g$. An alternating pairing $\langle, \rangle_\varphi{} : A_{\varphi} \times A_{\varphi} \rightarrow k$ is defined by
$$ \langle X_i, Y_j \rangle_{\varphi} = \delta_{i,j}, \; \; \; \;  \langle X_i, X_j \rangle_{\varphi} = \langle Y_i, Y_j \rangle_{\varphi} =0 \ \text{for} \ 1 \le i \le g.$$ 
\begin{defn} \label{def: N_varphi}
 The pair $(N_{\varphi}, \langle, \rangle_{\varphi})$ is called the \textit{standard type} defined by $\varphi$.
\end{defn}
Notably, as shown in \cite[Section 9.1]{oort2001stratification} there is a canonical filtration on $N_{\varphi}$ and $\text{ES}(N_{\varphi}) = \varphi$. Moreover, it follows from this construction and \cite[Theorem 9.4]{oort2001stratification} that  $\text{ES}$ is a bijection.

For any $ x = (X, \mu) \in \mathcal{A}_{g}$, $X[p]$ is a polarised $\text{BT}_{1}$ and we set $ \text{ES}(x) := \text{ES}(X[p])$.
\begin{defn}
 Let $\varphi$ be an elementary sequence of length $g$. The \textit{Ekedahl-Oort} stratum defined by $\varphi$ is
 \begin{center}
     $S_{\varphi} := \{ x \in \mathcal{A}_{g} \ : 
     \text{ES}(x) = \varphi \}$.
 \end{center}
\end{defn}
The Ekedahl-Oort strata form a good stratification of $\mathcal{A}_{g}$ (\cite[Theorem 1.3]{oort2001stratification}) and we recall key properties.  

\begin{prop}{(Oort \cite[Theorems 1,2 and p.56]{oort2001stratification})} \label{EOfacts} 
Let $g \in \mathbb{Z}_{\ge 1}$ and $\varphi$ an elementary sequence of length $g$. 
\begin{itemize}
    \item[(i)] $S_{\varphi}$ is non-empty, quasi-affine and its dimension is $\sum_{i=1}^{g} \varphi \left( i \right)$.
    \item[(ii)] $S_{\varphi}$ is locally closed.
    \item[(iii)] The $p$-rank of a stratum $S_{\varphi}$ is $\max \{ i \ : \ \varphi \left( i\right) =i \}$.
    \item[(iv)] The $a$-number of the stratum $S_{\varphi}$ is  $g - \varphi \left( g \right)$.
    \end{itemize}
\end{prop}

\section{Intersection Tools}
\label{intsec}
\subsection{Containment in $\mathcal{S}_g$, first slopes and $a$-numbers}
We collect several results from literature which allow us to conclude that certain intersections are non-empty.
\begin{prop}{({\rm Chai-Oort} \cite[Theorem 4.8]{chai2011monodromy})} \label{contaimentinsslocus} Let $\varphi$ be an elementary sequence of length $g$. Then 
$$ S_{\varphi} \subseteq \mathcal{S}_{g} \iff  \varphi\left(\floor{\frac{g+1}{2}}\right) =0.$$
\end{prop}

For other $p$-rank zero strata, we have an analogous result due to Harashita, which we summarise here. Associated to any elementary sequence $\varphi$ of length $g$ there is a rational number $\lambda_{\varphi}$, constructively defined as follows. Let $\psi$ be the final sequence stretched from $\varphi$ and define the map 
\begin{center}
 $\Phi : \{ 1, \ldots, 2g\} \rightarrow \{ 1, \ldots, 2g \}$   
\end{center}
by $ \Phi(i) = \psi(i)$ if $ \psi(i) \ne 0$ and $\Phi(i) = g + i$ otherwise. We set 
\begin{center}
    $D = \displaystyle \bigcap_{i=1}^{\infty} \text{Im}(\Phi^{i})$ and $C = D \cap \{ g+1, \ldots, 2g\}$
\end{center}
where $\Phi^{i}$ denotes the $i$th iterate of $\Phi$. Then $\lambda_{\varphi} = \vert C \vert / \vert D \vert $. 
\begin{defn}
The \textit{first Newton slope} of an elementary sequence $\varphi$ is the constant $\lambda_{\varphi}$ constructed as above.
\end{defn}
\begin{eg}
  If $\varphi = (0,1,2,3,4)$, we find 
  $$ \Phi : (1,2,3,4,5,6,7,8,9,10) \mapsto (6,1,2,3,4,4,4,4,4,5),$$
and $\text{Im}(\Phi^{i+5k}) = \{ 1,2,3,4,6\} $ for all $2 \le i \le 6$, $k$ a positive integer. Thus $\lambda_{\varphi} = 1/5$. We computed the first Newton slopes of all elementary sequences of length $5$ and p-rank $0$, and our computations are summarised in  Table \ref{tab:firstslopes}.
\end{eg}
\begin{table}
    \centering
    \begin{tabular}{|c|c|}
    \hline
    $\varphi$ &  $\lambda_{\varphi}$ \\
    \hline 
    $(0,1,2,3,3), (0,1,2,3,4) $ & $1/5$   \\
    $(0,1,2,2,2), (0,1,2,2,3) $ & $1/4 $ \\
    $ (0,0,1,2,2), (0,0,1,2,3), (0,1,1,1,1), (0,1,1,1,2),  (0,1,1,2,2), (0,1,1,2,3) $   & $1/3$ \\
    $(0,0,1,1,1), (0,0,1,1,2)  $  & $ 2/5 $ \\
    $ (0,0,0,0,0), (0,0,0,0,1), (0,0,0,1,1), (0,0,0,1,2)  $ & $1/2$ \\
    \hline     
    \end{tabular}
      \caption{First Newton slopes of Ekedahl-Oort strata in dimension $5$ and $p$-rank $0$}
    \label{tab:firstslopes}
    \end{table}
\begin{defn}
The \textit{first slope} of a Newton polygon is the smallest non-zero slope of the polygon.     
\end{defn}
For any $\lambda \in \mathbb{Q}$, Harashita defines a locus of $\mathcal{A}_g$
\begin{center}
    $Z_{\lambda} := \{ ( X, \mu) \in \mathcal{A}_{g}  \ \vert \ \text{the first slope of $X$ is at least} \  \lambda \}.$
\end{center}

\begin{prop}{({\rm Harashita} \cite[Corollary 4.2]{harashitaslope})} \label{firstslope}
Let $\varphi$ be an elementary sequence of length $g$ and  $\lambda \in \Q$. Then
$$S_{\varphi} \subseteq Z_{\lambda} \Leftrightarrow \lambda_{\varphi} \ge \lambda.$$
\end{prop}

In the case of dimension $5$ and $p$-rank zero, each Newton polygon has a unique first slope, and hence we will use the following observation to conclude that certain intersections are non-empty. 
\begin{cor} \label{firstslopecor}
Let $\varphi$ be an elementary sequence of length $g$. Suppose that there is a unique Newton stratum $\mathcal{N} \subset \mathcal{A}_g$  whose defining Newton polygon has first slope $\lambda_{\varphi}$. Then $S_{\varphi} \cap \mathcal{N} \ne \emptyset$.
\end{cor}
\begin{proof}
 Suppose that the intersection is empty. Then $S_{\varphi} \subset Z_{\lambda'}$ for some $\lambda' > \lambda_{\varphi}$, and this contradicts the definition of $Z_{\lambda}$ and Proposition \ref{firstslope}.
\end{proof}

Lastly, we state a result regarding intersections of Newton polygon strata and the Ekedahl-Oort stratum of $a$-number $1$. 
\begin{prop} \label{anumber1intersections}
Let $S_{\varphi} \subset \mathcal{A}_g$ be the unique Ekedahl-Oort stratum of $a$-number $1$ and $p$-rank $f$. For any Newton stratum $\mathcal{N} \subset \mathcal{A}_g$ of $p$-rank $f$, $ \mathcal{N} \cap S_{\varphi}$ is non-empty.    
\end{prop}
\begin{proof}
When $\mathcal{N} = \mathcal{S}_g$ and $f=0$, Li and Oort prove that $\mathcal{S}_g \cap S_{\varphi}$ is dense in $\mathcal{S}_g$ \cite[Theorem 4.9(iii)]{li1680f}. When $\mathcal{N} \ne \mathcal{S}_g$, Chai and Oort prove that $\mathcal{N}$ is geometrically irreducible (\cite[Theorem 3.1]{chai2011monodromy}) and the $a$-number of the generic point $(X, \lambda) \in \mathcal{N}$ is at most one by \cite[Remark 5.4]{oort2000newton}. Since $S_{\varphi}$ is the unique Ekedahl-Oort stratum of $p$-rank $f$ and $a$-number $1$, the result follows.
\end{proof}

\subsection{Inductive Arguments} \label{inducsec}
For two $p$-divisible groups $H$ and $H^{'}$, with corresponding Newton polygons $\xi$ and $\xi'$, the direct sum $H \oplus H'$ is canonically a $p$-divisible group, and its Newton polygon $\xi \oplus \xi'$ is obtained in the canonical manner. Similarly, for two polarised $BT_1$s, $G$ and $G'$, the direct sum $G \bigoplus G'$ is canonically a polarised $BT_1$. If $\varphi$ and $\varphi'$ are the corresponding elementary sequences, we write  $\varphi \oplus \varphi'$  for the elementary sequence of $G \oplus G'$. We make the following observation. 
\begin{prop} \label{inductiveprop}
  Let $\varphi$ and $\varphi'$ be two elementary sequences of length $g$ and $g'$, and $\mathcal{N}(\xi) \subseteq \mathcal{A}_{g}$, $ \mathcal{N}(\xi') \subseteq \mathcal{A}_{g'}$ two Newton strata such that $S_{\varphi} \cap \mathcal{N}(\xi) \neq \emptyset$ and $S_{\varphi'} \cap \mathcal{N}( \xi') \neq \emptyset$. Then $S_{\varphi \oplus \varphi'} \cap \mathcal{N}( \xi \oplus \xi')  \ne \emptyset$.
   \end{prop}
\begin{proof}
 This follows from the fact that $S_{\varphi} \times S_{\varphi'} \subseteq S_{\varphi \oplus \varphi'}$  and $\mathcal{N}(\xi) \times \mathcal{N}(\xi') \subseteq \mathcal{N}( \xi \oplus \xi')$.   
\end{proof}

To compute $ \varphi \oplus \varphi'$ in explicit examples we review several results from \cite[Section 4]{harashita2009configuration}. 
\begin{defn}
  A \textit{final type} of length $d$ is a pair $(B, \delta)$, where:
  \begin{itemize}
      \item $B = \{ b_{1} < \ldots < b_{d} \}$ is totally ordered finite set;
      \item $ \delta: B \rightarrow \{ 0, 1 \}$ is a map.
      \end{itemize}
 A final type $(B, \delta)$ is \textit{symmetric} if $ d= 2g$ for some $g \ge 1$ and it satisfies $\delta(b_{i}) + \delta(b_{2g +1 - i}) =1$ for $1 \le i \le g $. 
\end{defn}
For any final sequence $\psi$ we define a final type by setting $\delta(b_{i}) = 1 - \psi(i) + \psi( i-1)$, and note that this defines a bijection between the set of (symmetric) final sequences and the set of (symmetric) final types.  Let $\psi$ be a a final sequence of length $d$ and $(B, \delta)$ its associated final type, with $B = \{ b_{1} < \ldots < b_{d} \}$. We define a map $\pi_{\delta} : B \rightarrow B$ by 

\begin{center}
    $\pi_{\delta} \left( b_{i} \right)  = \left\{
	\begin{array}{ll}
		 b_{\psi(i)}   & \mbox{if } \delta(b_{i}) = 0; \\
		b_{\psi(d) + i - \psi(i)} & \mbox{if } \delta(b_{i}) =1.  \\
	\end{array}
\right.$
\end{center}
\begin{defn}
   Let $(B, \delta)$ be a final type, $C \subseteq B$ a subset and $ \epsilon = \delta \vert _{C}$. Then $(C, \epsilon)$ is a \textit{final subtype} of $(B, \delta)$ if $ \pi_{\delta}(C) = C$. 
\end{defn}
\begin{defn}
 A final type $(B, \delta)$ is called \textit{indecomposable if} it has no proper final subtype.    
\end{defn}
We write $D$ for the ordered subset of $[0,1]$ consisting of elements $u =  \sum_{l=1}^{\infty} a_{l}2^{-l} $ with $a_{l} \in \{0, 1 \} $ and with $a_{l+d} = a_{l}$ for all $l \ge 1$, for some $d$.  For an indecomposable final type $\mathcal{B} = (B, \delta)$, we define a map 
\begin{center}
$\nu_{\mathcal{B}} : B \rightarrow D
$, $b \mapsto \displaystyle \sum_{l=1}^{\infty } \delta( \pi_{\delta}^{- l}(b)) 2^{-l}$. 
\end{center}
The following two results allow us to determine $\varphi \oplus \varphi'$ for any two elementary sequences $\varphi$, $\varphi'$.

\begin{prop}{({\rm Harashita} \cite[Corollary 4.8 (i)]{harashita2009configuration})}
Let $\mathcal{C} = (C, \epsilon)$ be an  indecomposable final type and $e \ge 1$ an integer. Let $\mathcal{B}$ be the final type $\mathcal{C}^{\oplus e}$. Then $ \mathcal{B} = (B , \delta)$ is obtained as follows: 
\begin{center}
    $B = \{ c_{1}^{1} < \ldots < c_{1}^{e} < c_{2}^{1} < \ldots, < c_{2}^{e} < \ldots < c_{d}^{1} < \ldots < c_{d}^{e} \}$ and $\delta(c_{i}^{j}) = \epsilon^{j}(c_{i}^{j})$ for all $1 \le i \le d$,$ 1 \le j \le e$, 
    \end{center}
where we write  $ (C^{j}, \epsilon^{j}) = (C, \epsilon)$ and $C^{j} = \{ c_{1}^{j} < \ldots < c_{d}^{j} \}$ for $ 1 \le j \le e$.
    \end{prop}
\begin{eg} Fix an integer $e \ge 1$. If $\varphi = (0)$ , then $C = \{ c_1, c_2 \}$ and we find $\epsilon(c_1) =1, \epsilon(c_2) =0$ and hence $\delta(b_i) =1 $ and $\delta( b_{e+ i}) =0$ for all $ 1 \le i \le e$, and thus $\varphi^{\oplus e } = (0, \ldots, 0)$. If $\varphi = (1)$,  $\epsilon(c_1) = 0 , \epsilon(c_2) =1$ and hence 
$\delta(b_i) =0 $ and $\delta( b_{e+ i}) =1$ for $ 1 \le i \le e$, so $ \varphi^{\oplus e} = (1,2,\ldots, e )$.
\end{eg}
    
\begin{prop}{({\rm Harashita} \cite[Corollary 4.8 (ii)]{harashita2009configuration})}        
Let $\mathcal{C}_{j} = (C_{j}, \epsilon_{j})$ for $ 1 \le j \le d$ be powers of distinct indecomposable final types, with $ C_j = \{ c_{1,j} < \ldots < c_{d,j} \}$.
Then $\mathcal{B} = (B, \delta) = \bigoplus_{j=1}^{d} \mathcal{C}_{j}$ can be computed as follows: $B =  \union_{j} C_{j}$ with ordering 
    \begin{center}
          $ b < b'  \Leftrightarrow \left\{
	\begin{array}{ll}
		 b < b'   & \mbox{for} \ b,b' \in C_{j};\\
		\nu_{j} (b) < 
    \nu_{j'}(b') & \mbox{for} \ b \in C_{j}, b' \in C_{j'},  \\
	\end{array}
\right. $
\end{center}
where $\nu_j = \nu_{\mathcal{C}_{j}}$, and $\delta(c_{i,j}) = \epsilon_{j} (c_{i,j})$ for all $i$ and $j$.
\end{prop}
\begin{eg}
   If $\mathcal{C}_1$ corresponds to $\varphi =(0)$ and $\mathcal{C}_2$ corresponds to $\varphi' = (1)$, then we find $\pi_{\mathcal{C}_1}$ acts as the permutation $(1,2)$, whilst $\pi_{\mathcal{C}_2}$ is the identity, and further find: $ \nu_1(c_{1,1}) = \frac{1}{3}$, $ \nu_1(c_{2,1}) = \frac{2}{3}$,  $ \nu_2(c_{1,2}) = 0$, $ \nu_2(c_{2,2}) = 1$ and so 
   \begin{center}
       $ c_{1,2} < c_{1,1} < c_{2,1} < c_{2,2}$
   \end{center}
   with $ \delta = (0,1,0,1)$ and so $ \varphi \oplus \varphi' = (1,1)$.
   \end{eg}
To apply Proposition \ref{inductiveprop}, we require multiple additions, and in particular we need to understand which elementary sequences correspond to indecomposable final types. For this we computed all $\varphi \oplus \varphi'$ of total length at most $5$, which we computed using \texttt{magma}, and the functions \texttt{addition} and \texttt{mult} from the \texttt{generalfunctions.m} file of the online repository. The results of our computations for sequences of $p$-rank zero are summarised in Table \ref{additions}, those of positive $p$-rank appear in Table \ref{positiveprank}.

\begin{table}
\centering
    \begin{tabular}{|c|c|}
    \hline
   $g$  & Elementary Sequences \\ \hline 
   $1$ & $ \color{blue}{(0), (1)}$  \\ \hline 
   $2$   & $(0,0) = (0)^{\oplus 2}, \color{blue}{(0,1)}$  \\ \hline 

   $3$   & $(0,0,0) = (0)^{\oplus 3}$, $(0,0,1) = (0) \oplus (0,1), \color{blue}{(0,1,1),(0,1,2)} $   \\ \hline     
         $4$    & $(0,0,0,0) = (0)^{\oplus 4}$, $(0,0,0,1) = (0)^{\oplus 2} \oplus (0,1) $,$(0,0,1,2) = (0,1)^{\oplus 2}$, $(0,1,1,1) = (0) \oplus (0,1,1),$  \\

         &$(0,1,1,2) = (0) \oplus (0,1,2), \color{blue}{(0,0,1,1), (0,1,2,2), (0,1,2,3)} $ \\ \hline  
    
           & $(0,0,0,0,0) = (0)^{\oplus 5}, (0,0,0,0,1) = (0)^{\oplus 3} \oplus (0,1), (0,0,0,1,1) = (0) \oplus (0,0,1,1), $    \\ 
    $5$       & $(0,0,0,1,2) = (0) \oplus (0,1)^{\oplus 2}, (0,0,1,2,2) = (0,1) \oplus (0,1,1), (0,0,1,2,3) = (0,1) \oplus (0,1,2),   $ \\
           & $(0,1,1,1,1) = (0)^{\oplus 2} \oplus (0,1,1),(0,1,1,1,2) = (0)^{\oplus 2} \oplus (0,1,2), (0,1,2,2,2) = (0) \oplus (0,1,2,2),$  \\
           & $ (0,1,2,2,3) = (0) \oplus (0,1,2,3), \color{blue}{(0,0,1,1,1), (0,0,1,1,2), (0,1,2,3,3), (0,1,1,2,2),}$ \\ 
           & $ \color{blue}{(0,1,2,3,4), (0,1,1,2,3)}$ \\ \hline 
  \end{tabular}
      \caption{ Decompositions of elementary sequences of dimension $ 1 \le g \le 5$, indecomposable sequences are coloured blue.}
    \label{additions}
    \end{table}

\subsection{Minimal $p$-divisible groups}

For any Newton polygon $ \xi = \displaystyle \sum_{i=1}^{t} [m_{i}, n_{i}]$ we define a $p$-divisible group 
\begin{center}
    $G(\xi) :=  \displaystyle \bigoplus_{i=1}^{t} G_{m_{i},n_{i}}$
\end{center}
with $G_{m,n}$ defined as in Section \ref{npstrat}. For consistency, we note that the Newton polygon of the $p$-divisible group $G ( \xi)$ is  $\xi$. 
\begin{defn}
  A $p$-divisible group $\mathcal{G}$ is called \textit{minimal} if it is isomorphic to $G(\xi)$ for some Newton polygon $\xi$. 
\end{defn}

\begin{thm}[{\rm Oort} \cite{oort2005minimal}]
If $\mathcal{G}_{1}$ and $\mathcal{G}_{2}$ are $p$-divisible groups, defined over an algebraically closed field, and $\mathcal{G}$ is a minimal $p$-divisible group satisfying $\mathcal{G}_{1}[p] \isom \mathcal{G} [p]\isom \mathcal{G}_{2}[p]$, then $\mathcal{G}_{1} \isom \mathcal{G}_{2}$.
\end{thm}

Let $\xi$ be a Newton polygon. As shown in \cite[Proposition 3.7]{oort2004foliations}, there is a unique, up to isomorphism, principal quasi-polarisation on the $p$-divisible group $G(\xi)$, and thus to $G(\xi)[p]$ we associate a well-defined elementary sequence $\varphi_{\xi}$. 
\begin{cor} \label{minimalitycorollary}
 Let $\xi$ be a Newton polygon, $\varphi_{\xi}$ the associated elementary sequence and $\mathcal{N} (\xi)$ the Newton polygon stratum defined by $\xi$. Then $S_{\varphi_{\xi}} \subset \mathcal{N}( \xi) $.
\end{cor}
\begin{proof}
  If $ (X, \lambda) \in S_{\varphi_{\xi}}$, then $X[p] \isom G(\xi)[p]$, and since $G(\xi)$ is minimal, it follows that $X[p^{\infty}] \isom G(\xi)$ as $p$-divisible groups.  
\end{proof}
For a Newton polygon $\xi$, the associated elementary sequence $\varphi_{\xi}$ can be explicitly computed using \cite[Section 4.5]{harashita2009configuration}. We give a short overview of this here. 
For any coprime integers $m \ge n \ge0$, set
\begin{center}
$ \psi_{m,n} = (0, \ldots, 0, 1, \ldots , m )$ 
\end{center}
where the first $n$ entries are all $0$, and let $ \mathcal{B}_{m,n} = (B_{m,n}, \delta_{m,n})$ be its associated final type. This is an indecomposable final type by \cite[Corollary 4.15]{harashita2009configuration}.

\begin{table}[ht]
    \centering
    \begin{tabular}{|c|c|c|}
    \hline
    $\xi$ &  $\varphi_{\xi}$ & $p$-rank  \\
    \hline 
$2[1,0] + [2,1] + [1,2]  +2[0,1]$ & $(1,2,2,3,3) $ & $2$\\
$ 2[1,0] + 3[1,1] + 2[0,1]$ & $(1,2,2,2,2) $ & $2$ \\
$ [1,0] + [3,1] + [1,3] + [0,1]$ & $(1,1,2,3,3)$ & $1$ \\
$[1,0] + [2,1] + [1,1] + [1,2] + [0,1]$ & $(1,1,2,2,2) $ & $1$ \\
$[1,0] + 4[1,1] + [0,1] $ & $(1,1,1,1,1) $ & $1$ \\
   $[4,1] + [1,4]$ & $(0,1,2,3,3)$ & $0$ \\
   $[3,1] + [1,1] + [1,3]$ & $(0,1,2,2,2)$ & $0$ \\
   $[2,1] + 2[1,1] + [1,2]$ & $(0,1,1,1,1)$ & $0 $ \\
   $[3,2] + [2,3] $ & $(0,0,1,1,1)$ & $0$ \\
    $5[1,1] $ & $ (0,0,0,0,0)$ & $0$\\
  \hline 
   \end{tabular}
      \caption{Minimal elementary sequences in dimension $5$.}
    \label{minimaleltseq}
    \end{table}

\begin{prop}{({\rm Harashita} \cite[Corollary 4.22]{harashita2009configuration})} \label{finalseqcoms} Let $\eta = \sum_{i=1}^{t}[m_i, n_i]$ be a Newton polygon, with $n_i/(m_i + n_i) \le 1/2$ for all $i$, and let $\mathcal{B} = \bigoplus_{i=1}^{t} \mathcal{B}_{m_i, n_i}$ be the associated final type, and $\psi_{\eta}$ the associated final sequence.  Let $\xi = s[1,1] + \sum_{i=1}^{t} [m_i, n_i] + [n_i, m_i]$ be a symmetric Newton polygon and set $m = \sum_{i=1}^t m_i$, $n = \sum_{i=1}^t n_i$. Then $\varphi_{\xi}$ is given by 
\begin{center}
    $\varphi_{\xi} = (\psi_{\eta}(1), \ldots, \psi_{\eta}(m), m-n, \ldots, m-n)$
\end{center}
where the last $n+s$ entries are all $m-n$.
\end{prop}

\begin{eg}
 If $\xi = 2[1,0] + [2,1] + [1,2] + 2[1,0]$, then $\eta = 2[1,0] + [2,1]$, $m=4, n=1$ and $s=0$. Since $ \psi_{1,0} = (1)$ and $ \psi_{2,1} = (0,1,2)$, $\psi_{\eta}$  corresponds to $(1)^{\oplus 2} \oplus (0,1,2) = (1,2,2,3,4)$. Thus, by Proposition \ref{finalseqcoms}, it follows that $\varphi_{\xi} = (1,2,2,3,4)$.
 \end{eg}

 We computed all elementary sequences $\varphi_{\xi}$ for Newton polygons of dimension $5$ and $p$-rank $ 0 \le f \le 2$, and our calculations are summarised in Table \ref{minimaleltseq}.

\section{Positive $p$-rank intersections} \label{positivesections}
In this section we study intersections of positive $p$-rank.
\subsection{Decomposing Positive $p$-rank Intersections}
When there is a single Newton polygon of a fixed $p$-rank $f$, the $p$-rank $f$ stratum is itself a Newton stratum, and thus  all Ekedahl-Oort with $p$-rank $f$ automatically intersect this Newton stratum. This is the case in dimensions $1 \le g \le 3$ and positive $p$-rank, for dimension $4$ and $p$-rank $2 \le f \le 4$, and for dimension $5$ and $p$-rank $3 \le f \le 5$. All other intersections reduce to understanding intersections in lower dimensions of $p$-rank zero as we now explain. We begin with a preliminary result, which we believe to be well-known, but we include a short proof for the reader's convenience.

\begin{lem} \label{pdivgp}
Let $(Y,\mu)$ be a principally quasi-polarised $p$-divisible group of length $2g$ and dimension $g$. Then there exists a principally polarised abelian variety $(X,\lambda) \in \mathcal{A}_g$ and an isomorphism $X[p^\infty] \simeq Y$, sending $\lambda$ to $\mu$.
\end{lem}
\begin{proof}
Since $Y$ is principally quasi-polarised, its Newton polygon $\xi$ is symmetric, starting at $(0,0)$ and ending at $2g$. Then by \cite[page 98]{MR3077121}, there exists a principally polarised abelian variety $(X',\lambda') \in \mathcal{A}_g$ such that there is an isogeny $\phi: X'[p^\infty] \to Y$ which sends $\lambda_{X'[p^\infty]}$ to $\mu$. We let $X=X'/\ker \phi$ and we let $\lambda$ be the induced polarisation. Then this construction yields an isomorphism $X[p^\infty] \cong Y$ that is compatible with polarisations.
\end{proof}

\begin{prop} \label{positiveintersectionsprop}
    Let $\xi = f([1,0] + [0,1]) + \xi_0$ be a symmetric Newton polygon of dimension $g$ and positive $p$-rank $f$, and $\xi_0$ a symmetric Newton polygon of dimension $g_0 < g$ of $p$-rank $0$. Let $\varphi = (1)^{\oplus f} \oplus\varphi_0$ be an elementary sequence of length $g$ and $p$-rank $f$, and $\varphi_0$ an elementary sequence of length $g_0$ and $p$-rank $0$. Then 
    $$ S_{\varphi} \cap \mathcal{N}(\xi) \ne \emptyset \iff S_{\varphi_0} \cap \mathcal{N}(\xi_0) \ne \emptyset. $$
\end{prop}
\begin{proof} 
Suppose $(X,\lambda)\in S_\varphi \cap \mathcal{N}(\xi)$. Then the $p$-divisible group of $X$ decomposes as follows:
\[
X[p^\infty] \cong \left(E_k/(F-1,V) \oplus E_k/(F,V-1)\right)^f \oplus Y_0.
\]
Restricting the quasi-polarisation on $X[p^\infty]$ induced from $\lambda$ to $Y_0$ gives a principal quasi-polarisation $\lambda_0$ on $Y_0$. Then, by Lemma \ref{pdivgp} there exists $(X_0,\mu_0) \in \mathcal{A}_{g_0}$ such that $X_0[p^\infty] \cong Y_0$ with polarisation. It follows that $(X_0,\mu_0) \in S_{\varphi_0} \cap \mathcal{N}(\xi_0)$.

Conversely, $(1)^{\oplus f} = (1, \ldots, f)$ and $S_{(1, \ldots, f)} = \mathcal{N}(f([1,0] + [0,1]))$, so $ S_{\varphi} \cap \mathcal{N}(\xi) \ne \emptyset$ by Proposition \ref{inductiveprop}.    
\end{proof}

\subsection{Intersections in dimension at most $4$ and $p$-rank zero}
\label{lowerdims} In this section we review intersections for dimensions $1 \le g \le 4$ and $p$-rank zero. In dimensions $1$ and $2$ we have
\begin{center}
    $\mathcal{S}_{1} = S_{(0)}$ and $ \mathcal{S}_{2} = S_{(0,0)} \cup S_{(0,1)}$.
\end{center}
There are two Newton polygons of $p$-rank $0$ in dimension $3$, and understanding their intersections with the Ekedahl-Oort strata is a straightforward exercise.
\begin{prop}{(\cite[Example 3.5]{karemaker2024geometry})} \label{dimension3intersections}
The intersections of the Newton polygon and the Ekedahl-Oort strata in dimension $3$ and $p$-rank $0$ are completely described  by 
\begin{itemize}
    \item  $\mathcal{S}_{3} = S_{(000)} \cup S_{(001)} \cup ( S_{(012)} \cap \mathcal{S}_{3})$;
    \item $ \mathcal{N}([2,1] + [1,2] ) = S_{(011)} \cup ( S_{(012)} \cap \mathcal{N}([2,1] + [1,2])))$;
    \end{itemize}
and all of the above intersections are non-empty.
\end{prop}

In dimension $4$, intersections in $p$-rank $0$ are completely determined by Ibukiyama, Karemaker, and Yu.  

\begin{prop}{({ \rm Ibukiyama-Karemaker-Yu} \cite[Proposition 5.13]{ibukiyama2022polarised}) } \label{dim4intersections}
 The intersections of the Newton polygon and the Ekedahl-Oort stratifications in dimension $4$ and $p$-rank $0$ are completely described by 
 \begin{itemize}
     \item  $\mathcal{S}_{4} = S_{(0000)} \cup S_{(0001)} \cup S_{(0011)} \cup S_{(0012)} \cup (S_{(0112)} \cap \mathcal{S}_{4}) \cup ( S_{(0123)} \cap \mathcal{S}_{4});$
     \item $ \mathcal{N}([2,1] + [1,1] + [1,2])  = S_{(0,1,1,1)} \cup (S_{(0,1,1,2)} \cap \mathcal{N}([2,1] + [1,1] + [1,2])) \cup (S_{(0,1,2,3)} \cap \mathcal{N}([2,1] + [1,1] + [1,2])) ;$
     \item $\mathcal{N} ([3,1] + [1,3] ) =S_{(0,1,2,2)} \cup  ( S_{(0,1,2,3)} \cap \mathcal{N}([3,1] + [1,3])).$
     \end{itemize}
All intersections appearing above are non-empty. 
\end{prop}
\subsection{Positive $p$-rank intersections}
In dimension $5$, the only non-immediate intersections to consider are those of $p$-ranks one and two. It follows from Proposition \ref{positiveintersectionsprop} that this reduces to computing the appropriate sums of elementary sequence and apply the results of Propositions \ref{dimension3intersections} and \ref{dim4intersections}. We summarise our computations in Table \ref{positiveprank}.

\begin{table}
    \begin{tabular}{|c|c|c|c|}
    \hline 
    $g$ & $ f$ & $\xi$ & $ \varphi \ \text{such that} \ S_{\varphi} \cap \mathcal{N}(\xi) \ne \emptyset$  \\
    \hline 
    $5$ & $2$ & $2[1,0] + [2,1] + [1,2] + 2[0,1]$ & $(1,2,2,3,3) = (1)^{\oplus 2} \oplus (0,1,1) , (1,2,2,3,4) = (1)^{\oplus 2} \oplus (0,1,2) $ \\
    \hline 
    $5$ & $2$ & $2[1,0] + 3[1,1] + 2[0,1]$ & $(1,2,2,2,2) = (1)^{\oplus 2} \oplus (0,0,0), (1,2,2,2,3) = (1)^{\oplus 2} \oplus (0,0,1), $ \\ & & & $ (1,2,2,3,4) = (1)^{\oplus 2} \oplus (0,1,2)$ \\ 
    \hline 
    $5$ & $1$ & $[1,0] + [3,1] + [1,3] + [0,1]$ & $(1,1,2,3,3) = (1) \oplus (0,1,2,2), (1,1,2,3,4) = (1) \oplus (0,1,2,3)$\\
    \hline 
    $5$ & $1$ & $[1,0] + [2,1] + [1,1] + [1,2] + [0,1]$ & $(1,1,2,2,2) = (1) \oplus (0,1,1,1), (1,1,2,2,3) = (1) \oplus (0,1,1,2),$ \\ & & & $(1,1,2,3,4) = (1) \oplus (0,1,2,3)$\\
    \hline 
    $5$ & $1$ & $[1,0] + 4[1,1] + [0,1]$ & $(1,1,1,1,1) = (1) \oplus (0,0,0,0),(1,1,1,1,2) = (1) \oplus
    (0,0,0,1),$ \\ & & & $(1,1,1,2,2) = (1) \oplus (0,0,1,1), (1,1,1,2,3) = (1) \oplus (0,0,1,2), $ \\
    & & & $(1,1,2,2,3) = (1) \oplus (0,1,1,2), (1,1,2,3,4) = (1) \oplus (0,1,2,3)$ \\
    \hline 
   \end{tabular}
      \caption{Positive $p$-rank intersections in dimension $5$}
    \label{positiveprank}
    \end{table}

\section{Explicit Dieudonn\'{e} modules} \label{sec: explicit}

In this section, we settle some intersection problems by constructing suitable Dieudonn\'{e} modules, or showing that the relevant Dieudonn\'{e} modules do not exist.

\subsection{General approach} \label{ss: explicit general approach}

Recall the definition of a Dieudonn\'{e} module in Section~\ref{ss: Dd}. We say a Dieudonn\'{e} module is \emph{free} if it is free as a $W(k)$-module.
Via the Dieudonn\'{e} functor, principally quasi-polarised free Dieudonn\'{e} modules $M$ that have $W(k)$-rank $2g$ correspond to principally quasi-polarised $p$-divisible groups $G$ of length $2g$. If $(X,\lambda)$ is a principally polarised abelian variety of dimension $g$, then $(X[p^\infty],\lambda)$ is a principally quasi-polarised $p$-divisible group of length $2g$. Conversely, by Lemma \ref{pdivgp} a principally quasi-polarised $p$-divisible group $G$ of length $2g$ gives rise to a point $(X,\lambda)$ in $\mathcal{A}_g$ such that there is an isomorphism $X[p^\infty] \cong G$ that sends the polarisation of $X$ to the quasi-polarisation of $G$. 

It is useful for our purposes to explicitly construct $p$-divisible groups, because the $p$-divisible group determines both the Ekedahl-Oort stratum and the Newton stratum. If $\mathbb{D}(X[p^\infty])=M$, then the Ekedahl-Oort stratum of $X$ is determined by the isomorphism class of $M/pM$ and the Newton stratum of $X$ is determined by the isomorphism class of $M \otimes_{W(k)} W(k)[\frac{1}{p}]$.

Recall that $\sigma$ is the Frobenius automorphism of $W(k)$. Exhibiting a principally quasi-polarised free Dieudonn\'{e} module $M$ of $W(k)$-length $2g$ amounts to specifying
\begin{itemize}
    \item a $\sigma$-linear operator $F$ and $\sigma^{-1}$-linear operator $V$ on the $W(k)$-module $M^\circ :=W(k)\{a_1,\ldots a_{2g}\}$, satisfying $FV=VF=p$;
    \item a non-degenerate alternating bilinear pairing $\langle\cdot, \cdot\rangle$ on $M^\circ$, satisfying $ \langle F(x) , y \rangle = \langle x, V(y)\rangle^\sigma$ for every $x,y\in M^\circ$. 
    \end{itemize}
When $F$, $V$ and $\langle \cdot , \cdot \rangle$ are represented by matrices $A_F$, $A_V$ and $H$, respectively, the latter condition reduces to verifying $A_F^TH = (HA_V)^\sigma$.
    
We now outline how we compute the Ekedahl-Oort type of $M$. Given an element of $x\in M$, write $\overline{x}$ for its image in $M/pM$. We note that the set $\overline{B}=\{\overline{a_1}, \ldots ,\overline{a_{2g}}\}$ is a basis of $M^\circ/pM^\circ$ and, without loss of generality, we assume that $\overline{B}$ is preserved by $\overline{F}$ and $\overline{V}$, up to multiplication by $\pm 1$. Thus, $M/pM$ is, up to this multiplication by $\pm 1$, of the form $N_\varphi$ (see Definition~\ref{def: N_varphi}) for some elementary sequence $\varphi$. Since the occurrences of $\pm 1$ do not alter the Ekedahl-Oort type, but merely allow the polarisation to be defined over $\mathbb{F}_p$, we conclude that the elementary sequence of $M$ is $\varphi$.

For computing the Newton polygon, we use the following result of Zink.
\begin{lem}{({\rm Zink} \cite[Lemma 6.12]{Zink})}  \label{lem: Zink} 
For $n\in \mathbb{Z}_{\geq 1}$, define 
\[
s_M^{(n)} := \frac{\max \{m \; | \; F^n (M^\circ) \subseteq p^m M^\circ\}}{n}
\]
Then the first slope of the Newton polygon of $M$ is $s_M=\displaystyle \lim_{n\to \infty} s_M^{(n)}$
\end{lem}
Recall that the $\sigma$-linear operator $F$ on $M^\circ$ is represented by a matrix $A_F$. We will write $A_F^{(n)}=(f_{i,j}^{(n)})$ for the $n$-th $\sigma$-linear power. We then make the following observation.
\begin{lem} \label{lem: Zink+} 
For every $n\in \mathbb{Z}_{\geq 1}$, we have $s_M^{(n)} = \frac{1}{n}\min_{i,j} \left\{v_p\left(f_{i,j}^{(n)}\right)\right\}$ and $s_M \geq s_M^{(n)}$.
\end{lem}
\begin{proof}
The first statement is clear from the definition. For the second statement, suppose $s_M^{(n)}=m$. Then, by definition, $F^n(M^\circ) \subseteq p^m M^\circ$. As a consequence, we have, for every $k>0$, that $F^{nk}(M^\circ) \subseteq p^{mk} M^\circ$ and hence $s_M^{(nk)} \geq m$. The assertion then follows.
\end{proof}

In order to bound the first slope from above, we also use \cite[Theorem 1.2]{Nygaard} and \cite[1.5.1]{Katz}.

\subsection{Existence results}

In this section, we use the approach outlined above to prove that certain intersections of Ekedahl-Oort strata and Newton strata are non-empty. This amounts to exhibiting principally quasi-polarised free Dieudonn\'{e} modules with the appropriate mod-$p$ reduction and first slope.

At this point, it is useful to introduce the work of \cite{KraftpGruppen} (see also \cite[(9.9)]{oort2001stratification}). Given a word $w=u_{d-1} \cdots u_0$ in the alphabet $\{f,v\}$, one can assign to it a Dieudonn\'{e} module in the following way. Consider the vector space $k^d$ with basis $\{\overline{a_0}, \ldots \overline{a_{d-1}}\}$ and define the following actions of $\overline{F}$ and $\overline{V}$:
\[\overline{F}(\overline{a_i})=\begin{cases} \overline{a_{i+1 \bmod d}} &\hbox{if $u_i=f$} \\ 0 &\hbox{if $u_i=v$} \end{cases} \hspace{8mm}
\overline{V}(\overline{a_{i+1}}) = \begin{cases} \overline{a_{i \bmod d}} &\hbox{if $u_i=v$} \\ 0 &\hbox{if $u_i=f$.}
\end{cases}\]
We denote the Dieudonn\'{e} module thus obtained by $\overline{M}(w)$. We say $w$ is \emph{primitive} if it cannot be obtained by concatenating multiple copies of a shorter word.

\begin{thm}{({\rm Kraft} \cite{KraftpGruppen}, {\rm Oort }\cite[(9.9)]{oort2001stratification})} \label{thm: Kraft}
For any elementary sequence $\varphi$, we can write $N_\varphi \cong \bigoplus_{w\in W} \overline{M}(w)$, where $W$ is a unique multiset of primitive words. 
\end{thm}

Let $M$ be a free Dieudonn\'{e} module. As we shall see, knowing the multiset $W$ associated to $M/pM$ will help us compute the first slope of $M$. For $0<\lambda\leq 1/2$, let $\mathcal{N}_\lambda$ denote the unique Newton stratum of $\mathcal{A}_5$ with first slope $\lambda$.

\begin{lem} \label{lem: tautological}
The intersections $S_{(0,0,1,1,2)} \cap \mathcal{S}_5$, $S_{(0,1,1,2,2)} \cap \mathcal{S}_5$ and $S_{(0,1,1,2,3)} \cap \mathcal{N}_{2/5}$ are non-empty.
\end{lem}
\begin{proof}
In these cases, the relevant free Dieudonn\'{e} module $M$ is the \emph{tautological lift} of $N_\varphi$, as defined in \cite[(2.5)]{OortSimple}. We note that the polarisation in Definition~\ref{def: N_varphi} can also be lifted to a principal quasi-polarisation on $M^\circ$, simply by setting $\langle a_i, a_j\rangle = 1$ whenever $\langle \overline{a_i}, \overline{a_j} \rangle = 1$. One verifies that the condition $\langle F(x), y \rangle = \langle x,V(y) \rangle^\sigma$ is satisfied for any $x,y\in M^\circ$, which follows from the analogous condition that is satisfied modulo $p$. The Newton polygon of a tautological lift is calculated in \cite[(2.6)]{OortSimple}: for $w=u_{d-1} \ldots u_0$, the tautological lift of $\overline{M}(w)$ is isoclinic with slope $\frac{1}{d}\#\{0 \leq i \leq d-1 \; | \; u_i=v\}$, and the slopes of a direct sum are the slopes of the summands.

Using \cite[(9.1)]{oort2001stratification}, we compute $N_{(0,0,1,1,2)} \cong \overline{M}(ffvffvvfvv)$, hence the tautological lift has $1/2$ as its only slope. Through Lemma \ref{pdivgp}, it gives rise to a point in $S_{(0,0,1,1,2)} \cap \mathcal{S}_5$.

Similarly, one obtains $N_{(0,1,1,2,2)} \cong \overline{M}(fffvfvvvfv)$, yielding a point in $S_{(0,1,1,2,2)} \cap \mathcal{S}_5$.

Finally, one computes $N_{(0,1,1,2,3)} \cong \overline{M}(fffvv) \oplus \overline{M}(vvvff)$, giving a point in $S_{(0,1,1,2,3)} \cap \mathcal{N}_{2/5}$.
\end{proof}

\begin{lem} \label{lem: fffvv 1/2}
The intersection $S_{(0,1,1,2,3)} \cap \mathcal{S}_5$ is non-empty.
\end{lem}
\begin{proof}
We let $F$ act on $M^\circ = W(k) \{a_1,\ldots ,a_{10}\}$ as follows:
\begin{equation*}
\begin{split}
\begin{aligned}
    F(a_1)&=a_2 \\
    F(a_6)&=pa_7-pa_2
\end{aligned}
\end{split}
\quad \quad \quad
\begin{split}
\begin{aligned}
    F(a_2)&=a_3     \\
    F(a_7)&=pa_8
\end{aligned}
\end{split}
\quad \quad \quad
\begin{split}
\begin{aligned}
    F(a_3)&=a_4    \\
    F(a_8)&=pa_9 
\end{aligned}
\end{split}
\quad \quad \quad
\begin{split}
\begin{aligned}
    F(a_4)&=pa_5+pa_{10}     \\
    F(a_9)&=a_{10}
\end{aligned}
\end{split}
\quad \quad \quad
\begin{split}
\begin{aligned}
    F(a_5)&=pa_1      \\
    F(a_{10})&=a_6.
\end{aligned}
\end{split}
\end{equation*}

We let $V=pF^{-1}$. For $1\leq i\leq 5$, we set $\langle a_i, a_{j}\rangle=\delta_{i+5,j}$ and extend this to a $W(k)$-bilinear, alternating pairing. Using \texttt{magma} \cite{Magma}, we check that $A_F^TH=HA_V=(HA_V)^\sigma$ and conclude that this defines a principally quasi-polarised free Dieudonn\'{e} module $M$.

Reducing modulo $p$, we find $M/pM \cong \overline{M}(fffvv) \oplus \overline{M}(vvvff) \cong N_{(0,1,1,2,3)}$. 

The Newton polygon of $M$ is now computed through analysing the $p$-adic valuation of powers of $A_F$. Using \texttt{magma}, we establish that $p^9 \mid A_F^{(20)}$. Lemma~\ref{lem: Zink+} then yields that the first slope satisfies $s\geq 9/20$. The only eligible first slope in this range is $1/2$, so that we conclude that $M$ is supersingular.  Finally, through Dieudonn\'{e} theory and Lemma \ref{pdivgp}, $M$ gives rise to a point in $S_{(0,1,1,2,3)} \cap \mathcal{S}_5$.
\end{proof}

\begin{remark} \label{rmk: TL not highest}
In Lemma~\ref{lem: tautological}, we establish that the tautological lift of $N_{(0,1,1,2,3)}$ has slopes $2/5$ and $3/5$. On the other hand, in Lemma~\ref{lem: fffvv 1/2}, we construct a supersingular lift of $N_{(0,1,1,2,3)}$. Prior to discovering this, we were not aware of an example of a principally quasi-polarised free Dieudonn\'{e} module $M$ whose Newton polygon lies higher than that of the tautological lift of $M/pM$. Such examples do not exist in dimension $g \leq 4$.
\end{remark}

\begin{lem} \label{lem: (0,1,1,2,2) with 2/5}
The intersection $S_{(0,1,1,2,2)} \cap \mathcal{N}_{2/5}$ is non-empty.
\end{lem}
\begin{proof}
We now define the $\sigma$-linear operator $F$ on $M^\circ=W(k)\{a_1,\ldots a_{10} \}$ as follows:
\begin{equation*}
\begin{split}
\begin{aligned}
    F(a_1)&=a_2 \\
    F(a_6)&=pa_7
\end{aligned}
\end{split}
\quad \quad \quad
\begin{split}
\begin{aligned}
    F(a_2)&=a_3     \\
    F(a_7)&=pa_8
\end{aligned}
\end{split}
\quad \quad \quad
\begin{split}
\begin{aligned}
    F(a_3)&=a_4    \\
    F(a_8)&=pa_9 
\end{aligned}
\end{split}
\quad \quad \quad
\begin{split}
\begin{aligned}
    F(a_4)&=pa_5     \\
    F(a_9)&=a_{10}
\end{aligned}
\end{split}
\quad \quad \quad
\begin{split}
\begin{aligned}
    F(a_5)&=a_6+pa_1     \\
    F(a_{10})&=-pa_1.
\end{aligned}
\end{split}
\end{equation*}

As before, we let $V=pF^{-1}$ and set $\langle a_i,a_j\rangle=\delta_{i+5,j}$ for $1\leq i \leq 5$. We verify $A_F^TH=HA_V$ and conclude we have constructed a principally quasi-polarised free Dieudonn\'{e} module $M$. By reducing $A_F$ and $A_V$ modulo $p$, we establish $M/pM=\overline{M}(fffvfvvvfv) \cong N_{(0,1,1,2,2)}.$ Finally, we compute the Newton polygon of $M$. An explicit computation gives the block matrix
\[
A_F^{(5)} = \begin{bmatrix}
    \mathrm{diag}(p^2, p^2, p^2, p^2, p^2) & \mathrm{diag}(-p^4, -p^3, -p^2, -p, -p^2) \\
    \mathrm{diag}(p, p^2, p^3, p^4, p^3) & 0
\end{bmatrix}.
\]
We verify that $p^7 \mid A_F^{(20)}$ and hence $s_M\geq 7/20$ by Lemma~\ref{lem: Zink+}. On the other hand, inductively we show that  $v_p(f_{1,1}^{(5n)})=2n$ for every $n>0$, which implies $s_M \leq 2/5$. Hence we conclude $s_M=2/5$.

\end{proof}

\begin{lem} \label{lem: (0,0,1,2,3) with 2/5}
The intersection $S_{(0,0,1,2,3)} \cap \mathcal{N}_{2/5}$ is non-empty.
\end{lem}
\begin{proof}
Define the $\sigma$-linear operator $F$ on $M^\circ =W(k)\{a_1, \ldots ,a_{10}\}$ as follows:
\begin{equation*}
\begin{split}
\begin{aligned}
    F(a_1)&=a_2 \\
    F(a_6)&=a_7 
\end{aligned}
\end{split}
\quad \quad \quad
\begin{split}
\begin{aligned}
    F(a_2)&=a_3+pa_5 \\
    F(a_7)&=a_8+pa_1
\end{aligned}
\end{split}
\quad \quad \quad
\begin{split}
\begin{aligned}
    F(a_3)&=pa_4 \\
    F(a_8)&=pa_9
\end{aligned}
\end{split}
\quad \quad \quad
\begin{split}
\begin{aligned}
    F(a_4)&=-pa_1 \\
    F(a_9)&=pa_{10} 
\end{aligned}
\end{split}
\quad \quad \quad
\begin{split}
\begin{aligned}
    F(a_5)&=a_6 \\
    F(a_{10})&=-pa_5.
\end{aligned}
\end{split}
\end{equation*}
As before, we define $V=pF^{-1}$. We now set $\langle a_i, a_j \rangle = \delta_{i+2,j}$ for $1\leq i \leq 2$ and $\langle a_i,a_j \rangle = \delta_{i+3,j}$ for $5\leq i \leq 7$, which extends uniquely to a non-degenerate, bilinear, alternating pairing on $M^\circ$. We verify $A_F^TH =HA_V=(HA_V)^\sigma$. Reducing $A_F$ and $A_V$ modulo $p$ and using \cite[(9.1)]{oort2001stratification}, we see $M/pM \cong \overline{M}(ffvv) \oplus \overline{M}(fffvvv) \cong N_{(0,0,1,2,3)}$. Finally, we compute the Newton polygon of $M$. An explicit computation shows $p^7 \mid A_F^{(20)}$ and hence $s_M \geq 7/20$ by Lemma~\ref{lem: Zink+}. Furthermore, we verify that $p^9 \nmid f_{3,1}^{(22)}$, so that $M$ is not supersingular by \cite[Theorem 1.2]{Nygaard}. Thus the first slope of $M$ must be $2/5$.
\end{proof}

\begin{lem} \label{lem: (0,1,2,2,3) and 2/5}
The intersection $S_{(0,1,2,2,3)} \cap \mathcal{N}_{2/5}$ is non-empty.    
\end{lem}
\begin{proof}
Consider the following $\sigma$-linear operator $F$ on $M^\circ$:
\begin{equation*}
\begin{split}
\begin{aligned}
    F(a_1)&=a_2+pa_3 \\
    F(a_6)&=a_7 +pa_1
\end{aligned}
\end{split}
\quad \quad \quad
\begin{split}
\begin{aligned}
    F(a_2)&=-pa_1 \\
    F(a_7)&=pa_8
\end{aligned}
\end{split}
\quad \quad \quad
\begin{split}
\begin{aligned}
    F(a_3)&=a_4 \\
    F(a_8)&=pa_9
\end{aligned}
\end{split}
\quad \quad \quad
\begin{split}
\begin{aligned}
    F(a_4)&=a_5 \\
    F(a_9)&=pa_{10} 
\end{aligned}
\end{split}
\quad \quad \quad
\begin{split}
\begin{aligned}
    F(a_5)&=a_6 \\
    F(a_{10})&=-pa_3.
\end{aligned}
\end{split}
\end{equation*}
Set $V=pF^{-1}$. We define $\langle a_1,a_j\rangle =\delta_{2,j}$ and $\langle a_i,a_j\rangle = \delta_{i+4,j}$ for $3\leq i \leq 6$ and extend this to a bilinear alternating pairing, satisfying $A_F^T H = HA_V$. We observe $M/pM \cong \overline{M}(fv) \oplus \overline{M}(ffffvvvv) \cong N_{(0,1,2,2,3)}$. An explicit computation shows $p^7 \mid A_F^{(20)}$, which yields $s_M \geq 7/20$. Finally, we verify in this case that $p^9 \nmid f_{8,1}^{(22)}$, whence by \cite[Theorem 1.2]{Nygaard} $M$ is not supersingular and must have first slope $2/5$.
\end{proof}

\begin{remark} \label{rmk: superspecial}
The Dieudonn\'{e} module $M$ in the proof above has has positive superspecial rank, meaning $M/pM$ contains the $p$-torsion of a supersingular elliptic curve as a direct summand. On the other hand, the slope $1/2$ does not appear in the Newton polygon of $M$, meaning the $p$-divisible group of a supersingular elliptic curve cannot be embedded into $M$. This phenomenon does not occur in dimension $g \leq 4$.
\end{remark}

\subsection{Non-existence results}

We now show that certain intersections of Ekedahl-Oort strata and Newton strata are empty, by proving the non-existence of free Dieudonn\'{e} modules that simultaneously have the relevant mod-$p$ reduction and the relevant Newton polygon.

We first prove a general result, Proposition~\ref{prop: slope=1/n}, that might be of independent interest. For this, let $F$ be a $\sigma$-linear operator on $M^\circ=W(k)\{a_1, \ldots ,a_d\}$ and let $V=pF^{-1}$. Let $A_F^{(n)}=(f_{i,j}^{(n)})$ represent $F^n$ with respect to this basis.  Assume that the basis $\{\overline{a_1}, \ldots ,\overline{a_d}\}$ of $M^\circ/pM^\circ$ is preserved by $\overline{F}$ and $\overline{V}$. 

\begin{lem} \label{lem: p^2}
Let $1\leq i,j, \leq d$. Assume $\overline{a}_j \in \mathrm{im} (\overline{V})\setminus \{\overline{V}(\overline{a_i})\}$ and $\overline{a}_i \notin \mathrm{im}(\overline{F})$. Then $p^2 \mid f_{i,j}^{(1)}$.
\end{lem}
\begin{proof}
By assumption, there exists an integer $m\neq i$ such that $\overline{V}(\overline{a_m})=\overline{a_j}$, so that $V(a_m)=a_j + pb$ for some $b\in M^\circ$. Since $F\circ V=p$, we have 
$pa_m = F(V(a_m)) = F(a_j +pb) =F(a_j)+pF(b)$,
implying $F(a_j)=pa_m-pF(b)$, and in particular $F(a_j)\equiv pa_m - p F(b) \bmod p^2$.
It follows that $f_{i,j}^{(1)} \equiv 0 \bmod p^2$ when $\overline{a_i} \notin \mathrm{im}(\overline{F})$.
\end{proof}

\begin{prop} \label{prop: slope=1/n}
Let $M$ be a free Dieudonn\'{e} module and write $M/pM \cong \bigoplus_{\ell=1}^r \overline{M}(w_\ell)$ and assume that $w_1=ff \cdots fv$ has length $n$. Assume moreover that, for $2\leq \ell \leq r$, the word $w_\ell$ contains $v$ at least once and that any sequence of $f$'s in $w_\ell$ has length at most $n-1$. Then the first slope of the Newton polygon of $M$ is $\frac{1}{n}$.
\end{prop}
\begin{proof}
Let $a_1$ be the generator of $M^\circ$ so that $\overline{a_1} \in \overline{M}(w_1)$ and $\overline{a_1} \notin \mathrm{im}(\overline{F})$. Then by construction we have $v_p\left(f_{1,1}^{(n)}\right) = 1$. Moreover, by the assumption on the lengths of sequences of $f$'s in the other $w_\ell$, combined with Lemma~\ref{lem: p^2}, it follows that $v_p\left( f_{1,j}^{(n)}\right)\geq 2$ for every $j>1$. Note also that the assumptions imply $p \mid A_F^n$. Using induction, it then follows for every integer $k\geq 1$ that $p^k \mid A_F^{kn}$ and
$f_{1,1}^{(kn)} = \sum_{j} f_{1,j}^{(n)} f_{j,1}^{((k-1)n)}$
has valuation $k$, since the $j=1$ term has valuation $k$ (by the induction hypothesis) and the $j>1$ terms have valuation at least $k+1$. This implies that $s_M^{(nk)}=1/n$ for each $k$ and hence $s_M=1/n$.
\end{proof}

Note that $M$ here does not need to be quasi-polarised. We use this general result to settle the Ekedahl-Oort stratum $S_{(0,0,1,2,2)}$.

\begin{cor} \label{cor: (0,0,1,2,2) cont 1/3}
The Ekedahl-Oort stratum $S_{(0,0,1,2,2)}$ is completely contained in the Newton stratum $\mathcal{N}_{1/3}$.
\end{cor}
\begin{proof}
We begin by establishing $N_{(0,0,1,2,2)} \cong \overline{M}(ffv) \oplus \overline{M}(vvf) \oplus \overline{M}(ffvv)$. Now, let $(X,\lambda)$ be a point in $S_{(0,0,1,2,2)}$ and let $M=\mathbb{D}(X[p^\infty])$. Then applying Proposition~\ref{prop: slope=1/n} with $n=3$ and $w_1=ffv$ yields that the first slope of $M$ equals $1/3$, which determines the whole Newton polygon.
\end{proof}

\begin{lem} \label{lem: (0,1,1,1,2) not 2/5}
The intersection $S_{(0,1,1,1,2)} \cap \mathcal{N}_{2/5}$ is empty.
\end{lem}
\begin{proof}
Assume $M$ is a free Dieudonn\'{e} module such that $M/pM \cong N_{(0,1,1,1,2)} \cong \overline{M}(fv) \oplus \overline{M}(fv) \oplus \overline{M}(fffvvv)$. Without loss of generality, there exist $b_1, \ldots b_{10}\in M^\circ$ such that:
\begin{equation*}
\begin{split}
\begin{aligned}
    F(a_1)&=a_2+pb_1 \\
    F(a_6)&=a_7 +pb_6
\end{aligned}
\end{split}
\quad 
\begin{split}
\begin{aligned}
    F(a_2)&=-pa_1 +pb_2 \\
    F(a_7)&=a_8+pb_7
\end{aligned}
\end{split}
\quad
\begin{split}
\begin{aligned}
    F(a_3)&=a_4+pb_3 \\
    F(a_8)&=pa_9+pb_8
\end{aligned}
\end{split}
\quad 
\begin{split}
\begin{aligned}
    F(a_4)&=-pa_3+pb_4 \\
    F(a_9)&=pa_{10}+pb_9 
\end{aligned}
\end{split}
\quad 
\begin{split}
\begin{aligned}
    F(a_5)&=a_6 + pb_5 \\
    F(a_{10})&=-pa_5+pb_{10}.
\end{aligned}
\end{split}
\end{equation*}
for the purposes of this proof, we define $c:=f_{5,7}^{(1)}+f_{6,8}^{(1)}$ and note 
\begin{equation} \label{eq: c mod p^2}
    c=f_{5,7}^{(1)}+f_{6,8}^{(1)} \equiv f_{6,7}^{(2)} \equiv f_{6,6}^{(3)} \equiv f_{7,7}^{(3)} \equiv f_{8,5}^{(6)} \bmod p^2.
\end{equation}
We treat two cases based on $v_p(c)$.
\begin{enumerate}
\item First, assume $v_p(c) \geq 2$. 
Then a straightforward, but tedious, calculation, using Lemma~\ref{lem: p^2} and Equation~\eqref{eq: c mod p^2}, shows that $p^2 \mid A_F^{(6)}$ and $v_p(f_{i,j}^{(6)}) \geq 3$ if $i\in \{1,3,5,6,10\}$ or if $j\in\{2,4,7,8,9\}$.
Thus it follows that $p^5 \mid A_F^{(12)}$ and hence $s_M\geq 5/12$, implying $s_M=1/2$.

\item Now assume $v_p(c)=1$. Our aim is to show that $s_M<2/5$. For this, we use \cite[1.5.1]{Katz}, which provides that all Newton slopes of $M$ are at least $\lambda$ if and only if, for every $n$, we have $p^{\lceil n\lambda \rceil} | A_{F}^{(g+n)}$. Substituting $g=5$, $\lambda=2/5$ and $n=13$, we see that $M$ has first slope at least $2/5$ if and only if $p^6 \mid A_F^{(18)}$. We prove that $p^6 \nmid f_{8,5}^{(18)}$. We write $g_{k,\ell}:=f_{8,k}^{(6)} f_{k,\ell}^{(6)}f_{\ell,5}^{(6)}$ for $1\leq k,l \leq 10$, so that $f_{8,5}^{(18)}= \sum_{k,\ell} g_{k,\ell}$. We first observe that $f_{8,5}^{(6)}\equiv c \bmod p^2 \not \equiv 0$ and $v_p(f_{i,j}^{(6)})\geq 2$ for $(i,j)\neq (8,5)$, so that $p^6 \mid g_{l,\ell}$ unless $k=5$, $\ell=8$ or $(k,\ell)=(8,5)$. We furthermore note that $p^3 \mid f_{k,8}^{(6)}$ unless $k=8$, and, analogously, that $p^3 \mid f_{5,\ell}^{(6)}$ unless $\ell=5$. Hence $g_{k,\ell} \not \equiv 0 \bmod p^6$ only if $\{k,l\}=\{5,8\}$. We compute $f_{5,5}^{(6)} \equiv f_{5,7}^{(1)}c \bmod p^3$, while $f_{8,8}^{(6)} \equiv f_{6,8}^{(1)}c \bmod p^3$ and $f_{5,8}^{(6)}\equiv f_{5,7}^{(1)}f_{6,8}^{(1)} c \bmod p^4.$ We conclude
\begin{align*}
f_{8,5}^{(18)}&\equiv  f_{8,5}^{(6)}f_{5,5}^{(6)}f_{5,5}^{(6)} + f_{8,5}^{(6)}f_{5,8}^{(6)}f_{8,5}^{(6)} + f_{8,8}^{(6)}f_{8,5}^{(6)}f_{5,5}^{(6)} + f_{8,8}^{(6)}f_{8,8}^{(6)}f_{8,5}^{(6)} \\
&\equiv c^3((f_{5,7}^{(1)})^2+f_{5,7}^{(1)}f_{6,8}^{(1)}+f_{5,7}^{(1)}f_{6,8}^{(1)}+(f_{6,8}^{(1)})^2) \equiv c^5 \bmod p^6.
\end{align*}
Since $v_p(c)=1$, it follows that $p^6 \nmid f_{8,5}^{(18)}$ and hence $s_M<2/5$, meaning $s_M=1/3$.

\end{enumerate}
We conclude that $s_M$ can only be $1/2$ or $1/3$, and thus the intersection $S_{(0,1,1,1,2)} \cap \mathcal{N}_{2/5}$ is empty.

\end{proof}

\begin{remark} \label{rmk: Newton sandwich}
Combining Corollary~\ref{firstslopecor}, Proposition~\ref{anumber1intersections}, Table~\ref{additions} and Lemma~\ref{lem: (0,1,1,1,2) not 2/5}, one sees that $S_{(0,1,1,1,2)}$ intersects $\mathcal{N}_{1/3}$ and $\mathcal{S}_5$, but not $\mathcal{N}_{2/5}$, while the Newton polygon of $\mathcal{N}_{2/5}$ lies in between those of $\mathcal{N}_{1/3}$ and $\mathcal{S}_5$. In contrast, given any Ekedahl-Oort stratum $S_\varphi \subset \mathcal{A}_g$ for $g\leq 4$, the list of Newton strata intersecting $S_\varphi$ is determined by a lower bound and an upper bound for the Newton polygon. 
\end{remark}

\section{Intersections in dimension $5$ and $p$-rank $0$} \label{prankzeroproof}
To complete the proof of Theorem \ref{mainthm} we treat the case of $p$-rank zero. First, by Corollary \ref{minimalitycorollary} and the calculations in Table \ref{minimaleltseq} we find:
$$ S_{(0,1,2,3,3)} \subset \mathcal{N}_{1/5}, S_{(0,1,2,2,2)} \subset \mathcal{N}_{1/4}, S_{(0,1,1,1,1)} \subset \mathcal{N}_{1/3}, S_{(0,0,1,1,1)} \subset \mathcal{N}_{2/5},$$
and using Proposition \ref{contaimentinsslocus} we conclude
$$
 S_{\varphi} \subset \mathcal{S}_{g}  \Leftrightarrow \varphi \in \{ (0,0,0,0,0), (0,0,0,0,1), (0,0,0,1,1), (0,0,0,1,2) \}.
$$
By Proposition \ref{anumber1intersections}, we deduce that $S_{(0,1,2,3,4)}$ has non-empty intersection with all Newton strata. To decide the remaining intersections, we order by first slopes and use the following observation. 
\begin{prop} \label{orderingcor}
 Let $\varphi$ be an elementary sequence of length $g$ and $\xi$ a Newton polygon of dimension $5$. If $S_{\varphi} \cap \mathcal{N}(\xi) \ne \emptyset$, then the first slope of $\xi$ is at least $\lambda_{\varphi}$.
\end{prop}
\begin{proof}
 By Proposition \ref{firstslope} $S_{\varphi} \subset Z_{\lambda_{\varphi}}$  and our claim follows from the definition of $Z_{\lambda_{\varphi}}$. 
\end{proof}
From Table \ref{tab:firstslopes}, and the above remarks we deduce that 
$$\mathcal{N}_{1/5} = S_{(0,1,2,3,3)} \cup (S_{(0,1,2,3,4)} \cap  \mathcal{N}_{1/5}).$$
Similarly, 
$$\mathcal{N}_{1/4} = S_{(0,1,2,2,2)} \cup (S_{(0,1,2,2,3)} \cap \mathcal{N}_{1/4}) \cup (S_{(0,1,2,3,4)} \cap \mathcal{N}_{1/4})$$
and the second intersection appearing in this expression is non-empty by Corollary \ref{firstslopecor}.
By the above remarks, $\mathcal{N}_{1/3}$ fully contains $S_{(0,1,1,1,1)}$ and it has non-empty intersection with $S_{(0,1,2,3,4)}$. Moreover, by Corollary \ref{firstslopecor} it has non-empty intersections with the strata defined by
$$(0,1,1,2,3), (0,0,1,2,3), (0,1,1,2,2), (0,0,1,2,2), (0,1,1,1,2).$$
As $(0,1,2,2,3) = (0) \oplus (0,1,2,3)$, it follows from Propositions \ref{inductiveprop} and  \ref{dim4intersections} that $\mathcal{N}_{1/3} \cap S_{(0,1,2,2,3)} \ne \emptyset$. This completes the proof for $\mathcal{N}_{1/3}$ as there are no other possible non-empty intersections by Proposition \ref{orderingcor}.

For $\mathcal{N}_{2/5}$, we have $S_{(0,0,1,1,1)} \subset \mathcal{N}_{2/5}$, $S_{(0,1,2,3,4)} \cap \mathcal{N}_{2/5} \ne \emptyset$ and $S_{(0,0,1,1,2)} \cap \mathcal{N}_{2/5} \ne \emptyset$ by Corollary \ref{firstslopecor}. By Lemmata \ref{lem: tautological}, \ref{lem: (0,1,1,2,2) with  2/5}, \ref{lem: (0,0,1,2,3) with 2/5}, \ref{lem: (0,1,2,2,3) and 2/5} $\mathcal{N}_{\frac{2}{5}}$ has non-empty intersection with the strata defined by
$$(0,1,1,2,3), (0,1,1,2,2), (0,0,1,2,3),(0,1,2,2,3) $$
and by Corollary \ref{cor: (0,0,1,2,2) cont 1/3}  and Lemma \ref{lem: (0,1,1,1,2) not 2/5} it has empty intersection with the strata defined by 
$$ (0,0,1,2,2), (0,1,1,1,2).$$

It remains to treat the case of the supersingular locus $\mathcal{S}_5$. In addition to the intersections discussed above, $\mathcal{S}_5$ has non-trivial intersections with the strata defined by:
$$(0,1,1,1,2) = (0) \oplus (0,1,1,2), (0,1,2,2,3) = (0) \oplus (0,1,2,3),(0,0,1,2,3) = (0,1) \oplus (0,1,2)$$
by Propositions \ref{inductiveprop}, \ref{dimension3intersections} and \ref{dim4intersections}. The intersections of $\mathcal{S}_5$ with the strata defined by 
$$(0,0,1,1,2), (0,1,1,2,2) ,(0,1,1,2,3) $$
are non-empty by Lemmata \ref{lem: tautological} and \ref{lem: fffvv 1/2}, and $ S_{(0,0,1,2,2)} \cap \mathcal{S}_5 = \emptyset$ by Corollary \ref{cor: (0,0,1,2,2) cont 1/3}. This completes our proof.

\bibliographystyle{abbrv}

\bibliography{ref}

@incollection {oort2001stratification,
AUTHOR = {Oort, Frans},
     TITLE = {A stratification of a moduli space of abelian varieties},
 BOOKTITLE = {Moduli of abelian varieties ({T}exel {I}sland, 1999)},
    SERIES = {Progr. Math.},
    VOLUME = {195},
    editor = {Faber, Carel and van der Geer, Gerard and Oort, Frans},
     PAGES = {345--416},
 PUBLISHER = {Birkh\"auser, Basel},
      YEAR = {2001},
      ISBN = {3-7643-6517-X},
   MRCLASS = {14K10 (14L15)},
  MRNUMBER = {1827027},
MRREVIEWER = {Takashi\ Ichikawa},
       DOI = {10.1007/978-3-0348-8303-0\_13},
       URL = {https://doi.org/10.1007/978-3-0348-8303-0_13},
}

@article{harashitaslope,
  title={{Ekedahl-Oort strata and the first Newton slope strata}},
  author={Harashita, Shushi},
  journal={Journal of Algebraic Geometry},
  volume={16},
  number={1},
  pages={171},
  year={2007},
  publisher={Citeseer}
}

@article{harashita2009configuration,
  title={{Configuration of the central streams in the moduli of abelian varieties}},
  author={Harashita, Shushi},
journal={The Asian Journal of Mathematics},
volume={13},
number={2},
pages={215-250},
  year={2009}
}

@article{ibukiyama2022polarised,
  title={{When is a polarised abelian variety determined by its p-divisible group}},
  author={Ibukiyama, Tomoyoshi and Karemaker, Valentijn and Yu, Chia-Fu},
journal={Transactions of the American Mathematical Society, Series B},
  volume={12},
  number={03},
  pages={65--111},
  year={2025}
}

@article{oort2005minimal,
  title={{Minimal p-divisible groups}},
  author={Oort, Frans},
  journal={Annals of mathematics},
  volume={161},
  number={2},
  pages={1021--1036},
  year={2005},
  publisher={JSTOR}
}

@article{oort2004foliations,
  title={{Foliations in moduli spaces of abelian varieties}},
  author={Oort, Frans},
  journal={Journal of the American Mathematical Society},
  volume={17},
  number={2},
  pages={267--296},
  year={2004}
}

@article{yuri1963manin,
  title={{Theory of commutative formal groups over fields of finite characteristic}},
  author={Manin, Yuri},
  journal={Uspehi Mat. Nauk},
  volume={18},
  number={6},
  pages={114},
  year={1963}
}

@article{chai2011monodromy,
  title={{Monodromy and irreducibility of leaves}},
  author={Chai, Ching-Li and Oort, Frans},
  journal={Annals of mathematics},
  volume = {173},
  pages = {1359--1396},
  year={2011},
  publisher={JSTOR}
}

@article{karemaker2024geometry,
  title={Geometry and arithmetic of moduli spaces of abelian varieties in positive characteristic},
  author={Karemaker, Valentijn},
  journal={Abelian varieties, Lectures from the Arizona Winter School},
  year={2024}
}

@article{harashita2004number,
  title={{The a-number stratification on the moduli space of supersingular abelian varieties}},
  author={Harashita, Shushi},
  journal={Journal of Pure and Applied Algebra},
  volume={193},
  number={1-3},
  pages={163--191},
  year={2004},
  publisher={Elsevier}
}

@article{oort1991moduli,
  title={{Moduli of Abelian varieties and Newton Polygons}},
  author={Oort, Frans},
  journal={Comptes Rendus de l'Academie des Sciences-Series I: Mathematics},
  volume={312},
  number={5},
  pages={385--389},
  year={1991},
  publisher={Elsevier BV}
}

@book {li1680f,
    AUTHOR = {Li, Ke-Zheng and Oort, Frans},
     TITLE = {Moduli of supersingular abelian varieties},
    SERIES = {Lecture Notes in Mathematics},
    VOLUME = {1680},
 PUBLISHER = {Springer-Verlag, Berlin},
      YEAR = {1998},
     PAGES = {iv+116},
      ISBN = {3-540-63923-3},
   MRCLASS = {14K10 (11G10 14L05)},
  MRNUMBER = {1611305},
MRREVIEWER = {Ben\ Moonen},
       DOI = {10.1007/BFb0095931},
       URL = {https://doi.org/10.1007/BFb0095931},
}

@article{oort2000newton,
  title={{Newton polygons and formal groups: conjectures by Manin and Grothendieck}},
  author={Oort, Frans},
  journal={Annals of Mathematics},
  volume={152},
  number={1},
  pages={183--206},
  year={2000},
  publisher={JSTOR}
}

@article {ViehmannWedhorn,
    AUTHOR = {Viehmann, Eva and Wedhorn, Torsten},
     TITLE = {Ekedahl-{O}ort and {N}ewton strata for {S}himura varieties of
              {PEL} type},
   JOURNAL = {Math. Ann.},
  FJOURNAL = {Mathematische Annalen},
    VOLUME = {356},
      YEAR = {2013},
    NUMBER = {4},
     PAGES = {1493--1550},
      ISSN = {0025-5831,1432-1807},
   MRCLASS = {14G35 (11G18)},
  MRNUMBER = {3072810},
MRREVIEWER = {Marc-Hubert\ Nicole},
       DOI = {10.1007/s00208-012-0892-z},
       URL = {https://doi.org/10.1007/s00208-012-0892-z},
}

@article {OortSimple,
    AUTHOR = {Oort, Frans},
     TITLE = {Simple {$p$}-kernels of {$p$}-divisible groups},
   JOURNAL = {Adv. Math.},
  FJOURNAL = {Advances in Mathematics},
    VOLUME = {198},
      YEAR = {2005},
    NUMBER = {1},
     PAGES = {275--310},
      ISSN = {0001-8708,1090-2082},
   MRCLASS = {14L05 (14L15)},
  MRNUMBER = {2183256},
MRREVIEWER = {Alan\ Koch},
       DOI = {10.1016/j.aim.2005.03.015},
       URL = {https://doi.org/10.1016/j.aim.2005.03.015},
}

@unpublished{KraftpGruppen,
    title   = {Kommutative algebraische p-{G}ruppen (mit {A}nwendungen auf
p-divisible {G}ruppen und abelsche {V}ariet\"{a}ten)},
    author  = {Hans Peter Kraft},
    note    = {Manuscript},
    month   = {September},
    year    = {1975}
}

@article {Magma,
    AUTHOR = {Bosma, Wieb and Cannon, John and Playoust, Catherine},
     TITLE = {The {M}agma algebra system. {I}. {T}he user language},
      NOTE = {Computational algebra and number theory (London, 1993)},
   JOURNAL = {J. Symbolic Comput.},
  FJOURNAL = {Journal of Symbolic Computation},
    VOLUME = {24},
      YEAR = {1997},
    NUMBER = {3-4},
     PAGES = {235--265},
      ISSN = {0747-7171},
   MRCLASS = {68Q40},
  MRNUMBER = {MR1484478},
       DOI = {10.1006/jsco.1996.0125},
       URL = {http://dx.doi.org/10.1006/jsco.1996.0125},
}

@book {Zink,
    AUTHOR = {Zink, Thomas},
     TITLE = {Cartiertheorie kommutativer formaler {G}ruppen},
    SERIES = {Teubner-Texte zur Mathematik [Teubner Texts in Mathematics]},
    VOLUME = {68},
      NOTE = {With English, French and Russian summaries},
 PUBLISHER = {BSB B. G. Teubner Verlagsgesellschaft, Leipzig},
      YEAR = {1984},
     PAGES = {124},
   MRCLASS = {14L05},
  MRNUMBER = {767090},
MRREVIEWER = {Hans-J\"urgen\ Schneider},
}

@InProceedings{Nygaard,
author="Nygaard, Niels O.",
editor="Dolgachev, I.",
title="On supersingular abelian varieties",
booktitle="Algebraic Geometry",
year="1983",
publisher="Springer Berlin Heidelberg",
address="Berlin, Heidelberg",
pages="83--101",
isbn="978-3-540-40971-7"
}

@incollection{Katz,
     author = {Katz, Nicholas M.},
     title = {Slope filtration of $f$-crystals},
     booktitle = {Journ\'ees de G\'eom\'etrie Alg\'ebrique de Rennes - (Juillet 1978) (I) : Groupe formels, repr\'esentations galoisiennes et cohomologie des vari\'et\'es de caract\'eristique positive},
     series = {Ast\'erisque},
     pages = {113--163},
     year = {1979},
     publisher = {Soci\'et\'e math\'ematique de France},
     number = {63},
     mrnumber = {563463},
     zbl = {0426.14007},
     language = {en},
     url = {https://www.numdam.org/item/AST_1979__63__113_0/}
}

@incollection {MR3077121,
    AUTHOR = {Tate, John},
     TITLE = {Classes d'isog\'enie des vari\'et\'es ab\'eliennes sur un
              corps fini (d'apr\`es {T}. {H}onda)},
 BOOKTITLE = {S\'eminaire {B}ourbaki. {V}ol. 1968/69: {E}xpos\'es 347--363},
    SERIES = {Lecture Notes in Math.},
    VOLUME = {175},
     PAGES = {Exp. No. 352, 95--110},
 PUBLISHER = {Springer, Berlin},
      YEAR = {1971},
      ISBN = {3-540-05356-5; 0-387-05356-6},
   MRCLASS = {14K02},
  MRNUMBER = {3077121},
}

\end{document}